\documentclass[12pt,a4paper,oneside,reqno,english]{amsart}
\usepackage[T1]{fontenc}
\usepackage[utf8]{inputenc}
\usepackage{verbatim}
\usepackage{amsmath}
\usepackage{amstext}
\usepackage{amsthm}
\usepackage{stmaryrd}
\usepackage{setspace}
\makeatletter

\pdfpageheight\paperheight
\pdfpagewidth\paperwidth

\numberwithin{equation}{section}
\numberwithin{figure}{section}
\theoremstyle{plain}
\newtheorem{thm}{\protect\theoremname}[section]
\theoremstyle{remark}
\newtheorem{rem}[thm]{\protect\remarkname}
\theoremstyle{plain}
\newtheorem{lem}[thm]{\protect\lemmaname}
\theoremstyle{definition}

\usepackage{amsfonts}
\usepackage{amsthm}
\usepackage{epsfig}
\usepackage{mathrsfs}

\@ifundefined{definecolor}
 {\usepackage{color}}{}

\usepackage{bm}

\usepackage{graphicx}
\usepackage{rotating}
\usepackage{longtable}
\usepackage{pdflscape}
\usepackage{booktabs}
\usepackage{multirow}
\usepackage{float}

\usepackage[ruled,linesnumbered]{algorithm2e}

\usepackage[noabbrev,capitalise]{cleveref}

\makeatletter

\newcommand{\Rmnum}[1]{\expandafter\@slowromancap\romannumeral#1@}\makeatother

\numberwithin{equation}{section}

\allowdisplaybreaks

\newcommand{\abs}[1]{\left\vert#1\right\vert}
\newcommand{\set}[1]{\left\{#1\right\}}

\newcommand{\defs}{:=}
\newcommand{\sTo}{\rightarrow}

\newcommand{\dif}{\mathrm{d}}

\DeclareSymbolFont{lettersA}{U}{pxmia}{m}{it}
\DeclareMathSymbol{\piup}{\mathord}{lettersA}{"19}

\newcommand{\mr}[1]{\mathrm{#1}}

\makeatother

\makeatother

\usepackage{babel}
\providecommand{\lemmaname}{Lemma}
\providecommand{\problemname}{Problem}
\providecommand{\remarkname}{Remark}
\providecommand{\theoremname}{Theorem}

\begin{document}
\title[Uniqueness of 1-D Shock Solutions via Asymptotic Analysis]{On Uniqueness of Steady 1-D Shock Solutions in a Finite Nozzle
	via Asymptotic Analysis for Physical Parameters}
\author{Beixiang Fang}
\author{Su Jiang}
\author{Piye Sun}

\address{B.X. Fang: School of Mathematical Sciences, MOE-LSC, and SHL-MAC,
	Shanghai Jiao Tong University, Shanghai 200240, China }
\email{\texttt{bxfang@sjtu.edu.cn}}

\address{S. Jiang: School of Mathematical Sciences,
	Shanghai Jiao Tong University, Shanghai 200240, China}
\email{\texttt{js17854295180@sjtu.edu.cn}}

\address{P. Sun: School of Mathematical Sciences,
	Shanghai Jiao Tong University, Shanghai 200240, China}
\email{\texttt{sunpiye@amss.ac.cn}}



\date{\today}
\newpage
\begin{abstract}

	In this paper, we study the uniqueness of the steady 1-D shock solutions for the inviscid compressible Euler system in a finite nozzle via asymptotic analysis for physical parameters. The parameters for the heat conductivity and the temperature-depending viscosity are investigated for both barotropic gases and polytropic gases. It finally turns out that the hypotheses on the physical effects have significant influences on the asymptotic behaviors as the parameters vanish. In particular, the positions of the shock front for the limit shock solution( if exists ) are different for different hypotheses. Hence, it seems impossible to figure out a criterion selecting the unique shock solution within the framework of the inviscid Euler flows.

\end{abstract}

\maketitle
\tableofcontents{}

%

%

\section{Introduction}

In this paper we are going to study the uniqueness of 1-D steady transonic shocks for inviscid compressible flows in a finite nozzle via the asymptotic analysis for physical parameters, such as heat conductivity, viscosity, etc.
It is well-known that, for inviscid compressible flows in a finite nozzle governed by the steady 1-D Euler system, there exist infinite shock solutions with the position of the shock front being arbitrary, while the states both ahead of and behind it for these solutions being the same (see Figure \ref{fig:NormalShocks}).
All these steady transonic shock solutions satisfy the Rankine-Hugoniot conditions and the entropy condition.
However, uniqueness is expected for physical problems and one may wonder whether it is possible to figure out a physical criterion to select the unique physical shock solution among them.
In reality, physical effects ignored in the Euler system, such as heat conduction, viscous effects, etc., need to be taken into account and the motion of the fluid will be described by more complicated systems, for instance, Navier-Stokes equations.
Thus, the physical shock solution for the Euler system is naturally to be understood as the limit, if exists, of the solutions to the system with the physical effects as the associating parameters vanish.
See, for instance, \cite{BianchiniBressan2005AM,Gilbarg1951,GoodmanXin1992ARMA,HoffLiu1989IUMJ,HuangWangWangYang2015SCM,Wang2008AcMS,Yu1999ARMA} and references therein for literature on viscous shock profiles for Cauchy problems.
Then, based on this understanding, it is natural to ask whether it is true or not that such a limit does exist and, among all the transonic shock solutions for the steady 1-D Euler system, only one could be the limit.
Recently, in \cite{FangZhao2021CPAA}, Fang-Zhao shows that this is true for viscous flows in case that the viscosity is assumed to be constant. Namely, the viscous shock solutions converge as the viscosity vanishes, which yields that, among all previously mentioned shock solutions for the 1-D steady Euler system, only one can be the limit of steady viscous shock solutions as the viscosity goes to zero.
Then further questions naturally arise: Is this also true for other physical effects, namely, does there exist only one shock solution being the limit as the associating physical parameters vanish?
Moreover, if this is true, is the position of the shock front for the limit shock solution the same as the one obtained in \cite{FangZhao2021CPAA}?
Concerned with these questions, this paper will investigate the following two cases. 
\begin{enumerate}
	\item 
	The fluid with heat conduction. The coefficient of the heat conductivity will be assumed to be constant. This paper is trying to establish the existence of heat conductive transonic shock solutions with given conditions at the entrance and the exit of the nozzle, then to analyze the asymptotic behavior of the shock solutions as the coefficient of the heat conductivity vanishes. Two different types of the fluid will be discussed: barotropic gases and polytropic gases. It is interesting to observe in this paper that the heat conductive shock solutions for barotropic gases do not converge to any shock solution as the heat conductivity goes to zero, while the ones for polytropic gases converge to a shock solution. Morover, similar as in \cite{FangZhao2021CPAA}, for polytropic gases, the position of the shock front for the limit shock solution is uniquely determined, but it is different from the one determined in \cite{FangZhao2021CPAA}.
	\item The viscous fluid with the temperature-depending viscosity( i.e., the viscosity is assumed to depend on the temperature).
	This paper will establish the existence of the viscous transonic shock solutions with given conditions at the entrance and the exit of the nozzle, then analyze the asymptotic behavior of the shock solutions as the coefficient of the viscosity vanishes. Also, both barotropic gases and polytropic gases will be discussed. It will be shown that, for both types of fluids, the viscous shock solutions converge to a shock solution for the associated inviscid system. It is also observed that the position of the shock front for the limit shock solution depends strongly on the function describing the relation between the viscosity and the temperature, and is in general different from the one obtained in \cite{FangZhao2021CPAA} for the case of constant viscosity.
\end{enumerate}

\begin{figure}[th]
	\centering
	\def\svgwidth{240pt}
	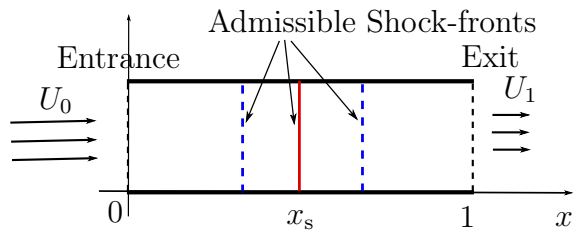
	
	\caption{Steady 1-D shocks for inviscid flows in a finite nozzle.} \label{fig:NormalShocks}
\end{figure}

\begin{rem}
	It is worth of pointing out that the results in this paper for the asymptotic behaviors of the transonic shock solutions with heat conduction or viscosity indicate that there are more than one shock solutions for the inviscid Euler system which can be physical solutions in the sense of being the limit solution as certain physical parameters vanish.
	Hence, it seems impossible to figure out a physical criterion for the steady 1-D Euler system which uniquely determines the shock solution.
\end{rem}

The mathematical analysis on gas flows with shocks in a finite nozzle plays a fundamental role in the theory for fluid dynamics and there have been plenty of studies on this issue. 
In \cite{CourantFriedrichs1948}, Courant and Friedriechs first gave a systematic analysis for various type of inviscid gas flows in a nozzle from viewpoint of nonlinear partial differential equations and special shock solutions had been established.
Since then, various nonlinear PDE models have been proposed and substantial progresses have been made on the well-posedness of these shock solutions.
In \cite{Liu1982ARMA} and \cite{Liu1982CMP}, for gas flows governed by the quasi-one-dimensional model, Liu showed that the shock occurs in the expanding portion of the nozzle is dynamically stable, while the shock in the contracting portion is not. 
For steady 2-D flows in an expanding nozzle, with appropriate pressure condition at the exit, the existence of the shock solution has been established in \cite{CourantFriedrichs1948} by assuming that the flow depends only on the radius parameter and the position of the shock front can be uniquely determined. The well-posedness of the shock solution has been established, for instance, in \cite{ChenYuan2013CPAA,Chen_S2009CMP,CuiYin2008JPDE,LiXinYin2009CMP,LiXinYin2011PJM,LiXinYin2013ARMA}.
While for steady 2-D flows in a finite flat nozzle, there exist infinite shock solutions and the position of the shock front can be arbitrary. By presuming that the shock front goes through a fixed point, given in advance artificially, the stability of the shock solution can also be established with certain selected condition imposed at the exit. For instance, see \cite{ChenFeldman2003JAMS,ChenChenSong2006JDE,XinYin2005CPAM} and references therein.
Without the assumption, in \cite{FangXin2021CPAM}, Fang and Xin develop new methods which successfully determine the position of the shock front with prescribed pressure condition, suggested by Courant and Friedrichs in \cite{CourantFriedrichs1948} at the exit of the nozzle, and establish the existence of the transonic shock solution.
One is also referred to, for instance, \cite{FangGao2021JDE,FangGao2022SIMA,FangLiuYuan2013ARMA,XinYanYin2011ARMA} and references therein for more literature on studies for well-posedness of transonic shocks in a finite nozzle.
By applying the methods developed in \cite{FangXin2021CPAM}, it is also observed that, for a generic perturbation of the flat nozzle with curved boundaries, there may exist multiple shock solutions with the same given pressure at the exit, similar with the phenomena observed in \cite{EmbidGoodmanMajda1984} for steady flows governed by quasi-one-dimensional model.
Hence, it seems impossible to determine the position of the shock front for steady flows in a flat nozzle within the framework of inviscid flows.
Since it is generally physically accepted that the inviscid flows can be regarded as limits of the viscous flows as the viscosity parameter goes to $ 0 $, it is natural to ask whether it is true or not that all shock solutions in the flat nozzle could be the limits of the viscous flows with the same boundary data.
One may referred to, for instance, \cite{BianchiniBressan2005AM,Gilbarg1951,GoodmanXin1992ARMA,GuesWilliams2002IUMJ,GuesMetivierWilliamsZumbrun2004CPAM,GuesMetivierWilliamsZumbrun2005JAMS,Wang2008AcMS,Yu1999ARMA} and references therein for literature on studies for vanishing viscous analysis of the shock profile in the whole space.
For one-dimensional flows in a finite interval, it has been established in \cite{BarkerMelinandZumbrun2019,EndresJenssenWilliams2010DCDS,MelinandZumbrun2019PhyD,Strani2014CPAA} the existence, uniqueness and stability of steady viscous shock solutions with certain inflow-outflow boundary data.
It is interesting that in \cite{FangZhao2021CPAA}, Fang and Zhao show that the viscous shock solutions for 1-D steady flows in a finite nozzle converge as the viscosity vanishes, which indicates that only one shock solution of the inviscid Euler system can be the limit of the viscous shock solutions.
In this paper, we are going to investigate what occurs for other physical effects as the associating parameters vanish.

The paper is organized as follows. In Section 2, the mathematical formulation of the singular limit problems will be described for fluids with heat conductivity or with  temperature-depending viscosity, respectively. The main results will also be described for each problem.
In Section 3, the asymptotic analysis for the heat conductive shock solutions are carried out.
It is shown in Section 3.1 that the heat conductive shock solutions for barotropic gases do not converge to weak solutions of the Euler system as the coefficient of the heat conductivity vanishes.
While it is proved in Section 3.2 that the heat conductive shock solutions for polytropic gases converge to a shock solution of the Euler system.
In Section 4, the asymptotic analysis for the temperature-depending viscous shock solutions are carried out.
In Section 4.1, it is shown that the temperature-depending viscous shock solutions for barotropic gases converge to a shock solution of the Euler system with the position of the shock depending on the parameter $ \delta $ in \eqref{vis_tem_function}.
In Section 4.2, it is observed that the asymptotic behaviors of the temperature-depending viscous shock solutions for polytropic gases strongly depend on $ \delta $ and are more complicated.
In particular, there may exist two different viscous shock solutions as $ \delta>1 $, which do not occur for barotropic gases.
It is shown that one of them converges to a shock solution of the Euler system and the other fails to converge to any weak solution of the Euler system.

\section{Mathematical Problems and Main Results}

In this section, infinite shock solutions for steady 1-D flow in a finite nozzle will firstly be described, then the associating singular limit problems and the results for fluids with heat conduction and with viscosity will be described in detail, respectively.
For fluids with heat conduction, it turns out that there is an essential difference between the asymptotic behavior for barotropic gases and the behavior for polytrotic gases: the heat conductive shock solutions for barotropic gases do not converge to any shock solution as the heat conductivity vanishes, while the ones for polytropic gases do converge to a shock solution for the Euler system without heat conduction. See Theorem \ref{thm: isen-heat} and Theorem \ref{thm: non-heat}.
For fluids with the temperature-depending viscosity, it will be shown that, for both barotropic gases and polytropic gases, the viscous shock solutions do converge to a shock solution for the associated inviscid system, and the position of the shock front depends strongly on the viscosity function of the temperature.


\subsection{Infinite Steady 1-D Shock Solutions for Inviscid Flows}\label{Section 2.1}

Without loss of generality, assume that the nozzle is bounded in the interval:
\begin{equation*}
	\mathcal{N}:=\{x:0<x<1\}.
\end{equation*}
Assume that the ideal fluid, ignoring the heat conductivity and the viscosity, in the nozzle is governed by the following steady 1-D Euler system under the assumption that the flow parameters depend only on the space variable $x\in(0,1)$,
\begin{equation}\label{aa}
	\left\{
	\begin{aligned}
		 & \partial_{x}(\rho u)        =0, \\
		 & \partial_{x}(\rho u^{2}+p)  =0, \\
		 & \partial_{x}(\rho u \Phi)   =0,
	\end{aligned}
	\right.
\end{equation}
where $u$ is the velocity, $(p,\rho)$ are the pressure and the density, $\Phi:=\frac{1}{2}u^{2}+i$ is the Bernoulli constant with $i=e+\frac{p}{\rho}$ the enthalpy and $e$ the internal energy. 
The fluid is assumed to be a polytropic gas such that the following state equation holds:
\begin{equation}\label{state}
	e=\frac{1}{\gamma-1}\cdot\frac{p}{\rho},
\end{equation}
or 
\begin{equation}\label{state_p}
	p=\mr{e}^{\frac{S}{c_{v}}}\rho^{\gamma},
\end{equation}
where $\gamma>1$ is the adiabatic exponent, $S$ is the entropy, and $c_{v}$ is the specific heat at constant volume.
Let $c:=\sqrt{\partial_{\rho}p}=\sqrt{\gamma p/\rho}$ be the sound speed and $M:=\displaystyle\frac{|u|}{c}$ be the Mach number, then the flow is supersonic in case $u>c$, or equivalently, $M>1$, and it is subsonic in case $u<c$, or equivalently, $M<1$.

Let $U=(u,\rho,p)$ represents the state of the flow. 
Then for a uniform supersonic constant state $U_{0}=(q_{0},\rho_{0},p_{0})$  there exists a unique subsonic state $U_{1}=(q_{1},\rho_{1},p_{1})$ which connects $ U_{0} $ via a shock, and across the shock front, there hold the following Rankine-Hugoniot conditions:
\begin{equation}\label{R-H}
	\begin{aligned}
		 & \rho_{0}q_{0}            =\rho_{1}q_{1} ,          \\
		 & \rho_{0}q^{2}_{0}+p_{0}  =\rho_{1}q^{2}_{1}+p_{1}, \\
		 & \Phi_{0}                 =\Phi_{1}.
	\end{aligned}
\end{equation}
Then it can be easily verified that, for any $0<x_{s}<1$, 
\begin{equation}\label{shock-solution}
	U^{0}(x;x_{s}):=
	\left\{
	\begin{aligned}
		U_{0}, & \quad 0\leq x < x_{s},  \\
		U_{1}, & \quad x_{s} < x \leq 1,
	\end{aligned}
	\right.
\end{equation}
gives a normal shock solution to the Euler system \eqref{aa}, with $x_{s}$ being the location of the shock front, which also satisfying the entropy condition $p_{1}>p_{0}$.
That is, for given supersonic state $ U_{0} $ at the entrance of the nozzle, 
there exist infinite shock solutions with the same state $ U_{1} $ behind the shock front, while the location of the shock front $x_{s}$ can be anywhere in the nozzle( see Figure \ref{fig:NormalShocks} ).
Since uniqueness is expected, one may wonder whether there exists a certain criterion which select the unique shock solution. 
In \cite{FangZhao2021CPAA}, Fang-Zhao shows that, among all the shock solutions, only one could be the limit of the viscous shock solutions for viscous flows with the viscosity being constant as the viscosity goes to zero.
Hence, we are wondering whether this is also true for other physical parameters. 
Moreover, if this is true, whether the positions of the shock front for the limit shock solutions coincide with the one in \cite{FangZhao2021CPAA}.
In this paper, we are going to investigate the shock solutions for fluids with heat conduction and with the viscosity depending on the temperature.

We remark that, for simplicity, without loss of generality, we always assume in the remaining part of the paper that
\begin{equation}\label{rho-velocity}
\rho_{0} q_{0}=1.
\end{equation}



\subsection{Singular limit problems for fluids with heat conduction}\label{Section 2.2}

In this subsection, the singular limit problems for fluids with heat conduction will be described in detail. The fluid will be assumed to be a polytropic gas with the state equation \eqref{state}, and a barotropic/isentropic gas with the following simplified state equation
\begin{equation}\label{state_baro}
p=\rho^\gamma.
\end{equation}
It is interesting to observe that the heat conductive shock solutions for barotropic gases do not converge to any shock solution as the heat conductivity goes to zero, while the ones for polytropic gases converge to a shock solution. 
Moreover, for polytropic gases, similar as in \cite{FangZhao2021CPAA}, the position of the shock front for the limit shock solution is uniquely determined, but it is different from the one determined in \cite{FangZhao2021CPAA}.
 

\subsubsection{The barotropic gases.}\label{Section 2.2.1}

The steady flow for barotropic gases with heat conduction is assumed to be governed by the following equations, together with the state equation \eqref{state_baro}, 
\begin{equation}\label{isen-heat-eq}
	\left\{
	\begin{aligned}
		 & \partial_{x}(\rho u)       =0,              \\
		 & \partial_{x}(\rho u \Phi)  =\epsilon e_{xx},
	\end{aligned}
	\right.
\end{equation}
where $\epsilon>0$ is the coefficient of the heat conductivity.
As $\epsilon$ goes to $0$, the system \eqref{isen-heat-eq} formally converges to the following system
\begin{equation}\label{isen-heat-eq_limit}
\left\{
\begin{aligned}
& \partial_{x}(\rho u)       =0,              \\
& \partial_{x}(\rho u \Phi)  =0,
\end{aligned}
\right.
\end{equation}
which is a simplified system of the Euler system \eqref{aa}.

Now let $U_{0}=(q_{0},\rho_{0},p_{0})$ with $p_{0} = \rho_{0}^{\gamma}$ be a supersonic state, then the subsonic state $U_{1}=(q_{1},\rho_{1},p_{1})$ with $p_{1} = \rho_{1}^{\gamma}$, connecting with $ U_0 $ via a shock, should satisfies the following R-H conditions:
\begin{equation}\label{R-H_heat}
\begin{aligned}
& \rho_{0}q_{0}            =\rho_{1}q_{1} ,          \\
& \Phi_{0}                 =\Phi_{1}.
\end{aligned}
\end{equation}
Thus, for any $0<x_{s}<1$, 
\begin{equation}\label{shock-solution_heat}
U_{hb}^{0}(x;x_{s}):=
\left\{
\begin{aligned}
U_{0}, & \quad 0\leq x < x_{s},  \\
U_{1}, & \quad x_{s} < x \leq 1,
\end{aligned}
\right.
\end{equation}
gives a normal shock solution to the simplified system \eqref{isen-heat-eq_limit}, with $x_{s}$ being the location of the shock front.
Then we are going to investigate the following singular limit problem as $\epsilon\sTo 0$, and try to figure out the ``physical'' shock solution among them.

\vskip 5px

\underline{Problem [{\bf HB}]}:

\begin{quotation}
	For any $\epsilon>0$, find the heat conductive shock solution $U = U_{hb}^{\epsilon}(x)$ to the system \eqref{isen-heat-eq} with the boundary conditions:
	\begin{equation}\label{isen-heat-bd}
	U(0)=U_{0},\qquad u(1)=q_{1},
	\end{equation}
	and find shock solutions, among solutions \eqref{shock-solution_heat} to the system \eqref{isen-heat-eq_limit}, which could be the limit of $ U_{hb}^{\epsilon}(x)$ ( or its subsequences ) as $\epsilon\sTo 0$. \qed
\end{quotation}

\vskip 5px

We follow the ideas in \cite{FangZhao2021CPAA} to deal with the problem [{\bf HB}]. However, the computations lead us to surprisingly negative results shown in the theorem below, which is different from the results in \cite{FangZhao2021CPAA}.




\begin{thm}\label{thm: isen-heat}
	For any $\epsilon>0$, there exists a unique solution $U=U_{hb}^{\epsilon}(x)\in C^{2}[0,1]$ to the equation (\ref{isen-heat-eq}) with boundary condition (\ref{isen-heat-bd}). 
	
	However, $U_{hb}^{\epsilon}$ do NOT converge to any shock solution of \eqref{isen-heat-eq_limit} expressed by \eqref{shock-solution_heat} as $\epsilon \to 0$, even in the weak sense.
\end{thm} 

\begin{rem}
	Theorem \ref{thm: isen-heat} indicates that NO shock solution in the form of \eqref{shock-solution} can be the limit of the heat conductive shock solutions $U_{hb}^{\epsilon}$ to the system \eqref{isen-heat-eq} with the boundary conditions \eqref{isen-heat-bd}  as $\epsilon \to 0$.
	Hence, it seems that NO shock solution in the form of  \eqref{shock-solution} could be the ``physical'' solution to the system \eqref{isen-heat-eq_limit}.
	It needs further investigation to clarify the asymptotic behaviors of $U_{hb}^{\epsilon}$  as $\epsilon \to 0$ and what does it means.
\end{rem}

\subsubsection{The polytropic gases.}
The steady flow for polytropic gases with heat conduction is assumed to be governed by the following equations, together with the state equation \eqref{state},
\begin{equation}\label{non-heat-eq}
	\left\{
	\begin{aligned}
		 & \partial_{x}(\rho u)        =0,               \\
		 & \partial_{x}(\rho u^{2}+p)  =0,               \\
		 & \partial_{x}(\rho u \Phi)   =\epsilon e_{xx},
	\end{aligned}
	\right.
\end{equation}
where $\epsilon>0$ is the coefficient of the heat conductivity.
As $\epsilon$ goes to $0$, the system \eqref{non-heat-eq} formally converges to the
Euler system \eqref{aa}.

Let $ U^{0}(x;x_{s}) $ defined by \eqref{shock-solution} be shock solutions to the Euler system \eqref{aa}, with $0<x_{s}<1$ being the location of the shock front.
Then we are going to study the following singular limit problem as $ \epsilon\sTo 0 $ to find the ``physical'' shock solution among them.

\vskip 5px

\underline{Problem [{\bf HP}]}:

\begin{quotation}
	For any $\epsilon>0$, find the heat conductive shock solution $U = U_{hp}^{\epsilon}(x)$ to the system \eqref{non-heat-eq} with the boundary conditions:
	\begin{equation}\label{non-heat-bd}
	U(0)=U_{0},\qquad p(1)=p_{1},
	\end{equation}
	and find shock solutions, among solutions \eqref{shock-solution} to the system \eqref{aa}, which could be the limit of $ U_{hp}^{\epsilon}(x)$ ( or its subsequences ) as $\epsilon\sTo 0$. \qed
\end{quotation}

\vskip 5px

We still follow the ideas in \cite{FangZhao2021CPAA} to deal with the problem [{\bf HP}].
Then degeneracy may occur in the coefficients of the principle part of the equation, which brings new difficulties in solving the boundary value problem \eqref{non-heat-eq}-\eqref{non-heat-bd}.
In this paper, we give a sufficient condition to avoid the degeneracy such that the existence of the heat conductive shock solution $U_{hp}^{\epsilon}(x)$ can be established.
Then the techniques in \cite{FangZhao2021CPAA} can also be employed to establish the convergence of $U_{hp}^{\epsilon}(x)$ as $\epsilon\sTo 0$.
We are going to prove the following theorem.

\begin{thm}\label{thm: non-heat}
	Suppose that the state equation \eqref{state} and the Mach number 
	$$M_{0}=\displaystyle \sqrt{\frac{\rho_{0}q_{0}^{2}}{\gamma p_{0}}}$$ 
	of the flow at the entrance of the nozzle satisfy either one of the following two conditions:
	\begin{enumerate}
		\item[(HP1)] $1<\gamma<3$, and 
			\begin{equation}\label{Mach}
			1<M_{0}^{2}\leq \frac{3\gamma-1}{\gamma(3-\gamma)};
			\end{equation}
		\item[(HP2)] $\gamma \geq 3$, and $ M_0>1 $.
	\end{enumerate}
	Then for any $\epsilon>0$, there exists a unique solution $U=U_{hp}^{\epsilon}(x)\in C^{2}[0,1]$ to the system (\ref{non-heat-eq}) with the boundary conditions (\ref{non-heat-bd}).

	Furthermore, $U_{hp}^{\epsilon}(x)$ converges to $\hat{U}_{hp}^{0}$ in $L^1(0,1)$ as $\epsilon \to 0$,
	where
	\begin{equation}
		\hat{U}^{0}_{hp}:=
		\left\{
		\begin{aligned}
			U_{0}, & \quad 0\leq x < X_{shp},  \\
			U_{1}, & \quad X_{shp} < x \leq 1,
		\end{aligned}
		\right.
	\end{equation}
	with the shock location
	\begin{equation}
		X_{shp}:=\displaystyle\frac{\gamma+1}{2(\gamma-1)}\cdot\frac{M_{0}^{2} \gamma-1}{M_{0}^{2} \gamma+1},
	\end{equation}
	is a normal shock solution to the equations \eqref{aa} among the ones \eqref{shock-solution}.
\end{thm}

It is worth of pointing out the following three remarks associating with Theorem \ref{thm: non-heat}.

\begin{rem}
	The assumption that either (HP1) or (HP2) holds is only a sufficient condition which is needed technically in the proof of Theorem \ref{thm: non-heat} to avoid degeneracy of the coefficient in the principle part of the reformulated equations. One is referred to Subsection 3.2 for more details. It needs further investigation whether the results of Theorem \ref{thm: non-heat} still hold as both (HP1) and (HP2) fail.
\end{rem}

\begin{rem}
	Theorem \ref{thm: non-heat} indicates that, under the condition that either (HP1) or (HP2) holds, there exists a unique heat conductive shock solution $ U_{hp}^{\epsilon}(x) $ to the boundary value problem \eqref{non-heat-eq}-\eqref{non-heat-bd} for every heat conductivity $ \epsilon>0 $, and they converge to a shock solution among \eqref{shock-solution} to the Euler system \eqref{aa} as $ \epsilon $ goes to zero.
	This fact is essentially different from the results obtained in Theorem \ref{thm: isen-heat} for barotropic gases, in which the heat conductive shock solutions do NOT converge to any shock solution as $ \epsilon\sTo 0 $.
	The reasons why such a difference appears is a quite interesting topic worth to be further studied.
\end{rem}

\begin{rem}
	Similar as the results obtained in \cite{FangZhao2021CPAA} for viscous fluids with constant viscosity, Theorem \ref{thm: non-heat} shows that there exists a shock solution $\hat{U}_{hp}^{0}$ among \eqref{shock-solution} which is the limit of the heat conductive shock solutions $ U_{hp}^{\epsilon} $ as $ \epsilon $ vanishes.
	However, it can be easily checked that the shock solution $\hat{U}_{hp}^{0}$ is different from the one obtained in \cite{FangZhao2021CPAA}.
	Namely, different physical effects may converge to different shock solutions as the associating physical parameters vanish.
	Such phenomena will also be observed below for fluids with temperature-depending viscosity.
	Hence, it seems that the asymptotic analysis for physical parameters do NOT give unified uniqueness criterion selecting the ``physical'' shock solutions among \eqref{shock-solution}.
\end{rem}


\subsection{Singular limit problems for fluids with temperature-depending viscosity}\label{Section 2.3}

In this subsection, the singular limit problems for fluids with temperature-depending viscosity will be described in detail.
Also, the fluid will be assumed to be a barotropic/isentropic gas with the simplified state equation \eqref{state_baro} and a polytropic gas with the state equation \eqref{state}, respectively.
Moreover, the temperature $ T $ satisfies $p=\rho R T$, where $R$ is the constant for ideal gases.
For both types of fluids, the viscosity $\tilde{\mu}$ depends on the temperature $ T $, which is described by the following function: 
\begin{equation}\label{vis_tem_function}
\tilde{\mu}(T)=\mu T^\delta,
\end{equation}
where $\delta>0$ is a fixed constant and $\mu>0$ is a small constant which is called the coefficient of temperature-depending viscosity in this paper.
We are going to investigate the asymptotic behavior of the steady shock solutions in a finite nozzle as $ \mu\sTo 0 $.
Note that in case $\delta=0$, the viscosity does not depend on temperature and it is a constant, which is exactly the case that has been studied in \cite{FangZhao2021CPAA}.
As $\delta>0$, similarly as in \cite{FangZhao2021CPAA}, it will be shown in this paper that the viscous shock solutions converge to a shock solution as $ \mu\sTo 0 $. 
However, it will be observed that the position of the shock front for the limit shock solution depends on the parameter $ \delta $.
Namely, it depends on a parameter in the viscosity function \eqref{vis_tem_function} which goes beyond the parameters for the inviscid Euler system.
Hence, it seems that one cannot simply determine the ``physical'' shock solution via the vanishing viscosity analysis.
Moreover, it is observed that there may exist two different viscous shock solutions for polytropic gases, which will not occur for barotropic gases.
The non-uniqueness of the viscous shock solutions results in more complicated computations and analysis.


\subsubsection{The barotropic gases.}

Under the assumption \eqref{vis_tem_function} for the viscosity, the steady flow for barotropic/isentropic gases with temperature-depending viscosity is assumed to be governed by the following equations, together with the state equation \eqref{state_baro},
\begin{equation}\label{isen-viscous-eq}
	\left\{
	\begin{aligned}
		 & \partial_{x}(\rho u)        =0,                                        \\
		 & \partial_{x}(\rho u^{2}+p)  =\mu \partial_{x}(T^{\delta}\partial_{x}u). \\
	\end{aligned}
	\right.
\end{equation}
As the coefficient of temperature-depending viscosity $ \mu $ goes to $ 0 $, the system \eqref{isen-viscous-eq} formally converges to the following system
\begin{equation}\label{isen-viscous-eq_limit}
\left\{
\begin{aligned}
& \partial_{x}(\rho u)        =0,                                        \\
& \partial_{x}(\rho u^{2}+p)  =0, \\
\end{aligned}
\right.
\end{equation}
which is a simplified system of the Euler system \eqref{aa}, and is usually called the steady isentropic Euler system.

Now let $U_{0}=(q_{0},\rho_{0},p_{0})$ with $p_{0} = \rho_{0}^{\gamma}$ be a supersonic state, then the subsonic state $U_{1}=(q_{1},\rho_{1},p_{1})$ with $p_{1} = \rho_{1}^{\gamma}$, connecting with $ U_0 $ via a shock, should satisfies the following R-H conditions:
\begin{equation}\label{R-H_vis}
\begin{aligned}
& \rho_{0}q_{0}            =\rho_{1}q_{1} ,          \\
& \rho_{0}q^{2}_{0}+p_{0}  =\rho_{1}q^{2}_{1}+p_{1}.
\end{aligned}
\end{equation}
Thus, for any $0<x_{s}<1$, 
\begin{equation}\label{shock-solution_vis}
U_{vb}^{0}(x;x_{s}):=
\left\{
\begin{aligned}
U_{0}, & \quad 0\leq x < x_{s},  \\
U_{1}, & \quad x_{s} < x \leq 1,
\end{aligned}
\right.
\end{equation}
gives a normal shock solution to the simplified system \eqref{isen-viscous-eq_limit}, with $x_{s}$ being the location of the shock front.
Then we are going to investigate the following singular limit problem as $\mu\sTo 0$, and try to figure out the ``physical'' shock solution among them.

\vskip 5px

\underline{Problem [{\bf VB}]}:

\begin{quotation}
	For any $\mu>0$, find the viscous shock solution $U = U_{vb}^{\mu}(x)$ to the system \eqref{isen-viscous-eq} with the boundary conditions:
	\begin{equation}\label{isen-viscous-bd}
		U(0)=U_{0},\qquad u(1)=q_{1}.
	\end{equation}
	and find shock solutions, among solutions \eqref{shock-solution_vis} to the system \eqref{isen-viscous-eq_limit}, which could be the limit of $ U_{vb}^{\mu}(x)$ ( or its subsequences ) as $\mu\sTo 0$. \qed
\end{quotation}

\vskip 5px


The ideas and techniques in \cite{FangZhao2021CPAA} can also be employed to deal with problem [{\bf VB}], which leas us to the following theorem.
\begin{thm} \label{thm: isen-viscous}
	For any $\mu>0$, there exists a unique solution $U=U_{vb}^{\mu}(x)\in C^{2}[0,1]$ to the equations (\ref{isen-viscous-eq}) with boundary conditions (\ref{isen-viscous-bd}).

	Furthermore, $U^{\mu}_{vb}$ converges to $\hat{U}_{vb}^{0}$ in $L^1(0,1)$ as $\mu \to 0$,
	where
	\begin{equation}
		\hat{U}^{0}_{vb}:=
		\left\{
		\begin{aligned}
			U_{0}, & \quad 0\leq x < X_{svb},  \\
			U_{1}, & \quad X_{svb} < x \leq 1,
		\end{aligned}
		\right.
	\end{equation}
	with the shock location
	\begin{equation}\label{shock_location_vb}
		X_{svb}:=\left(1+(\frac{q_{0}}{q_{1}})^{\delta-\gamma-1}\frac{q_{0}^{\gamma+1}-\gamma}{\gamma-q_{1}^{\gamma+1}}\right)^{-1}
		=\left(1+(\frac{M_0}{M_1})^\frac{2\delta}{\gamma+1} \frac{1-{M_0}^{-2}}{M_1^{-2}-1}\right)^{-1}.
	\end{equation}
	is a normal shock solution to the equations \eqref{isen-viscous-eq_limit} among the ones \eqref{shock-solution_vis}
\end{thm}

\begin{rem}
	This result is similar to the one for barotropic case in \cite{FangZhao2021CPAA}, with a difference on the shock location of the limit solution. 
	It can be easily observed in \eqref{shock_location_vb} that the position of the shock front depends on the value of the parameter $\delta$ in \eqref{vis_tem_function}.
	As $\delta$ becomes greater, the shock will be closer to the entrance of the nozzle. 
\end{rem}

\subsubsection{The polytropic gases.}

Under the assumption  \eqref{vis_tem_function} for the viscosity, the steady flow for polytropic gases with temperature-depending viscosity is assumed to be governed by the following equations, together with the state equation \eqref{state_p},
\begin{equation}\label{non-viscous-eq}
	\left\{
	\begin{aligned}
		 & \partial_{x}(\rho u)        =0,                                           \\
		 & \partial_{x}(\rho u^{2}+p)  =\mu \partial_{x}(T^{\delta}\partial_{x}u),   \\
		 & \partial_{x}(\rho u \Phi)   =\mu \partial_{x}(T^{\delta}u\partial_{x} u).
	\end{aligned}
	\right.
\end{equation}
As the coefficient of temperature-depending viscosity $ \mu $ goes to $ 0 $, the system \eqref{non-viscous-eq} formally converges to the Euler system \eqref{aa}.

Let $ U^{0}(x;x_{s}) $ defined by \eqref{shock-solution} be shock solutions to the Euler system \eqref{aa}, with $0<x_{s}<1$ being the location of the shock front.
Then we are going to study the following singular limit problem as $ \mu\sTo 0 $ to find the ``physical'' shock solution among them.

\vskip 5px

\underline{Problem [{\bf VP}]}:

\begin{quotation}
	For any $\mu>0$, find the viscous shock solution $U = U_{vp}^{\epsilon}(x)$ to the system \eqref{non-viscous-eq} with the boundary conditions:
	\begin{equation}\label{non-viscous-bd}
		U(0)=U_{0},\qquad u(1)=q_{1}.
	\end{equation}
	and find shock solutions, among solutions \eqref{shock-solution} to the system \eqref{aa}, which could be the limit of $ U_{vp}^{\mu}(x)$ ( or its subsequences ) as $\mu\sTo 0$. \qed
\end{quotation}

\vskip 5px

Again, the ideas and techniques in \cite{FangZhao2021CPAA} are employed to deal with the problem [{\bf VP}].
However, it turns out that the existence and uniqueness of $C^2$ viscous shock solutions to the problem \eqref{non-viscous-eq}-\eqref{non-viscous-bd} depends on the value of the parameter $\delta$ in \eqref{vis_tem_function}. 
The following theorem will be proved in this paper.
\begin{thm}\label{thm: non-viscous}
	Suppose
	\begin{equation}\label{mu-gamma}
		1<M_{0}^{2}<\frac{2\gamma}{\gamma-1}.
	\end{equation}

\begin{enumerate}
	\item If $0 < \delta \leq 1$, then there exists a small constant $\mu_1>0$ such that,
	for any $\mu \in (0, \mu_1)$, there exists a unique solution $U=U_{vp}^{\mu}(x)\in C^{2}[0,1]$ to the boundary value problem (\ref{non-viscous-eq})-(\ref{non-viscous-bd}).
	
	Furthermore, $U^{\mu}_{vp}$ converges to $\hat{U}_{vp}^{0}$ in $L^1(0,1)$ as $\mu \to 0$,
	where $\hat{U}_{vp}^{0}$ is one of the normal shock solutions \eqref{shock-solution} to the system \eqref{aa} with the following expression
	\begin{equation}\label{U_0vp}
	\hat{U}^{0}_{vp}:=
	\left\{
	\begin{aligned}
	U_{0}, & \quad 0\leq x < X_{svp},  \\
	U_{1}, & \quad X_{svp} < x \leq 1,
	\end{aligned}
	\right.
	\end{equation}
	where
	\begin{align}\label{shock_location_vp}
	X_{svp}:&= \left( 1+ \left(\frac{\frac{\gamma+1}{\gamma-1} - \frac{q_1}{q_0}}{\frac{\gamma+1}{\gamma-1} - \frac{q_0}{q_1}}\right)^{\delta}
	\frac{q_1}{q_0} \right)^{-1} \\
	&=\left( 1+\left(\frac{M_0}{M_1}\right)^{2\delta}
	\left(\frac{1-M_0^{-2}}{M_1^{-2}-1} \right)^{1+2\delta}
	\right)^{-1},
	\end{align}
	represents the position of the shock front.
	
	\item If $\delta>1$, then there exists a small constant $\mu_2>0$ such that, for any $\mu \in (0,\mu_2)$,
	there exist two viscous shock solutions $U_{vp1}^\mu(x), U_{vp2}^\mu(x) \in C^2[0,1]$ to the  boundary value problem (\ref{non-viscous-eq})-(\ref{non-viscous-bd}).
	
	Furthermore, only one of the two solutions, say, $U_{vp1}^\mu$ converges to $\hat{U}_{vp}^{0}$ in $L^{1}(0,1)$ as $\mu\to 0$, where $\hat{U}_{vp}^{0}$, defined with the same expression as in \eqref{U_0vp}, is one of the normal shock solutions \eqref{shock-solution} to the system \eqref{aa}.
\end{enumerate}


\end{thm}

\begin{rem}
	For polytropic gases, an interesting thing is that the uniqueness of the solutions will be effected by the value of $\delta$. As $\delta \leq 1$, the uniqueness holds and the viscous shock solutions converge to an inviscid shock solution. 
	However, as $\delta > 1$, for each small $\mu$, there exist two viscous shock solutions. One of the solutions converges to a shock solution as $ \mu\sTo 0 $, while the other does not converge to a weak solution of the Euler system \eqref{aa}. 
	Thus, there arises another interesting problem on the uniqueness of the viscous shock solutions to the boundary value problem  (\ref{non-viscous-eq})-(\ref{non-viscous-bd}) as $\delta > 1$, which needs further investigation. 
\end{rem}


\begin{rem}
	Theorem \ref{thm: non-viscous} indicates that the position of the shock front for the inviscid shock solution, which is the very limit of the viscous shock solutions as the coefficient of the temperature-depending viscosity $ \mu $ goes to $ 0 $, strongly depends on the value of the parameter $ \delta $ in \eqref{vis_tem_function}.
	Hence, there are more than one shock solutions among \eqref{shock-solution} for the inviscid Euler system \eqref{aa} can be physical solutions in the sense of being the limit solution as the viscosity vanishes.
	Moreover, since $ \delta $ is a parameter beyond the ones for inviscid flows, it seems impossible to figure out a physical criterion for the steady 1-D Euler system which uniquely determines the shock solution.
\end{rem}

Before finishing this section, we give a note on notations in this paper.
\vskip 5pt

\noindent NOTE: For simplicity of the notations, we will drop the subscript of $*_{hb},*_{hp},*_{vb},*_{vp}$ as there is no confusion taking place in the argument hereafter, since the notations of each case are independent.

\section{Asymptotic Behaviors as the Heat Conductivity Vanishes}

This section is devoted to analyze the asymptotic behaviors of the steady shock solutions for fluids with heat conduction.
We are going to follow the techniques developed in \cite{FangZhao2021CPAA} to prove Theorem \ref{thm: isen-heat} for barotropic gases and Theorem \ref{thm: non-heat} for polytropic gases, respectively.

\subsection{The barotropic gases}\label{Section 3.1}

In this subsection, we firstly investigate the case of barotropic gases and deal with the problem [{\bf HB}].
Following the ideas in \cite{FangZhao2021CPAA}, the boundary value problem \eqref{isen-heat-eq} and \eqref{isen-heat-bd} will be reformulated as a problem for an ordinary equation of first order with an unknown parameter, which is studied carefully to lead us to Theorem \ref{thm: isen-heat}.


\subsubsection{Reformulation of the problem [{\bf HB}]} \label{Section 3.1.1}


Let $U^{\epsilon}(x):=(u^{\epsilon}(x),\rho^{\epsilon}(x),p^{\epsilon}(x))$, with $ p^{\epsilon}(x)=(\rho^{\epsilon}(x))^{\gamma} $, be a solution to the system \eqref{isen-heat-eq}, then it holds that, by the assumption \eqref{rho-velocity},
\begin{equation}\label{density-u}
	\rho^{\epsilon}=\frac{1}{u^{\epsilon}}.
\end{equation}
Substituting (\ref{density-u}) into the second equation in the system $(\ref{isen-heat-eq})$, one obtains that
\begin{equation}\label{hb-second-order}
	\partial_{x}\left(\frac{1}{2}(u^{\epsilon})^{2}+\frac{\gamma}{\gamma-1}\frac{1}{(u^{\epsilon})^{\gamma-1}}\right)=\frac{\epsilon}{\gamma-1}\partial_{xx}\left(\frac{1}{(u^{\epsilon})^{\gamma-1}}\right).
\end{equation}
Denoted by $\displaystyle \kappa:=\frac{\epsilon}{\gamma-1}$, $\displaystyle v^{\kappa}:=\displaystyle\frac{1}{(u^{\epsilon})^{\gamma-1}}$, the equation \eqref{hb-second-order} can be rewritten as
\begin{equation}\label{second-order}
	\partial_{x}g(v^{\kappa})=\kappa \partial_{xx}v^{\kappa},
\end{equation}
where
\begin{equation}\label{hb-g}
	g(v):=\frac{1}{2}v^{\frac{2}{1-\gamma}}+\frac{\gamma}{\gamma-1}v.
\end{equation}
Moreover, the boundary conditions \eqref{isen-heat-bd} is rewritten as 
\begin{equation}\label{second-order-b.c.}
v^{\kappa}(0)=v_{0},\qquad v^{\kappa}(1)=v_{1},
\end{equation}
where $v_{0}=q_{0}^{1-\gamma}$, and $v_{1}=q_{1}^{1-\gamma}$.
Note that, by the Rankine-Hugoniot conditions \eqref{R-H_heat}, it holds that
\begin{equation}\label{hb-g-bd}
	g(v_{0})=g(v_{1}).
\end{equation}
Thus, the BVP \eqref{isen-heat-eq} and \eqref{isen-heat-bd} is reduced to a boundary value problem of a nonlinear ordinary differential equation \eqref{second-order} of second order with the boundary conditions \eqref{second-order-b.c.} satisfying \eqref{hb-g-bd}.


Let $f(v)\defs g(v)-g(v_{0})$. Then it is a smooth, strictly convex function defined in $(0,\infty)$, and satisfies
\begin{equation}
f(v_{0})=f(v_{1})=0.
\end{equation}
Thus, it holds that
\begin{equation}\label{sign-f}
f(v)<0, \qquad \text{for} \quad \text{any} \quad v \in (v_{0},v_{1}),
\end{equation}
and there exists a unique $\displaystyle v_{*}:=\gamma^{\frac{1-\gamma}{1+\gamma}} \in (v_{0},v_{1}) $ such that $f'(v_{*})=0$. 
Obviously, one has
\begin{equation}\label{bg_inf_f}
	\inf_{v_{0}<v<v_{1}} f(v) = f(v_{*}) <0.
\end{equation}
Integrating (\ref{second-order}) over the interval $(0,x)$, one obtains
\begin{equation}\label{first-order}
	\kappa \partial_{x}v^{\kappa}=F(v^{\kappa};\alpha_{\kappa}),
\end{equation}
where $F(v;\alpha)\defs f(v)+\alpha$, and $\alpha_{\kappa}\defs\kappa \partial_{x}v^{\kappa}(0)$ is an unknown constant which needs to be determined together with $v^{\kappa}$ under the boundary conditions (\ref{second-order-b.c.}). 
That is, the BVP \eqref{second-order} and \eqref{second-order-b.c.} is further reduced to a boundary value problem of first order ordinary differential equation \eqref{first-order} and the boundary conditions \eqref{second-order-b.c.} for unknowns $ (v^{\kappa};\alpha_{\kappa}) $.

Moreover, the equation (\ref{first-order}) with heat conductivity formally converges to
\begin{equation}\label{weak-sol1}
F(v^{0}(x),\alpha_{0})=0
\end{equation}
as $\kappa \to 0$. The weak solution $(v^{0},\alpha_{0})$ to \eqref{weak-sol1} is defined in the sense that, for any test functions $\phi(x)\in C^{\infty}[0,1]$, it holds that
\begin{equation}\label{isen-weak-sol}
	\int_{0}^{1}F(v^{0}(x),\alpha_{0})\phi(x)dx = 0.
\end{equation}
Obviously, for any $ 0 < x_{s} <1 $, the weak solutions $(v^{0}(x;x_{s}), \alpha_{0})$ associating with the shock solutions \eqref{shock-solution_heat} satisfy that $ \alpha_{0} = 0 $ and
\begin{equation}\label{shock-solution_hb}
v^{0}(x;x_{s}):=
\left\{
\begin{aligned}
v_{0}, & \quad 0\leq x < x_{s},  \\
v_{1}, & \quad x_{s} < x \leq 1.
\end{aligned}
\right.
\end{equation}

Thus, the singular limit problem [{\bf HB}] is reformulated as the following problem.

\vskip 5px

\underline{Problem [{\bf HB-R}]}:

\begin{quotation}
	
	Let $\kappa>0$. Try to find a solution $(v^{\kappa},\alpha_{\kappa})$ satisfying the equation (\ref{first-order}) and the boundary conditions (\ref{second-order-b.c.}).
	
	Furthermore, find shock solutions, among weak solutions $(v^{0}(x;x_{s}), \alpha_{0})$ to the system \eqref{weak-sol1}, which could be the limit of $(v^{\kappa},\alpha_{\kappa})$ ( or its subsequences ) as $\kappa\sTo 0$. \qed
\end{quotation}

\vskip 5px


Therefore the aim of this subsection is to solve the reformulated problem [{\bf HB-R}]. 
The existence of the solutions $(v^{\kappa},\alpha_{\kappa})$ for any $\kappa>0$ will be established, while it will be proved that, as $\kappa\sTo 0$, they do not converge to any shock solution among  weak solutions $(v^{0}(x;x_{s}), \alpha_{0})$ to the system \eqref{weak-sol1}.


\subsubsection{Existence of the solution $(v^{\kappa},\alpha_{\kappa})$ to (\ref{first-order}) and  (\ref{second-order-b.c.})}\label{Section 3.1.2}

Motivated by the techniques in \cite{FangZhao2021CPAA}, the existence of the solution $(v^{\kappa},\alpha_{\kappa})$ to (\ref{first-order}) and  (\ref{second-order-b.c.}) will be established by solving the problem for the inverse function of $ v^{\kappa} $. 
Then the following lemma asserting the monotonicity of $ v^{\kappa} $ is needed to guarantee the existence of its inverse function.

\begin{lem}\label{isen-heat-prior-hopf}
	Let $\kappa>0$ and $v^{\kappa}\in C^{2}(0,1)\bigcap C^{1}[0,1]$ be a solution to the boundary value problem (\ref{second-order}) with the boundary condition (\ref{second-order-b.c.}). Then it holds that
	\begin{equation}\label{epsilon-isen-prior}
	v_{0}<v^{\kappa}(x)<v_{1},\qquad for \quad any \quad x\in (0,1).
	\end{equation}
	Moreover, there holds
	\begin{equation}\label{epsilon-isen-hopf}
	\partial_{x}v^{\kappa}(x)>0, \quad \text{for}\quad \text{any} \quad x\in [0,1].
	\end{equation}
\end{lem}
\begin{proof}
	\eqref{epsilon-isen-prior} can be obtained by applying the maximum principles for elliptic equations of second order. On the other hand, due to Hopf lemma, it is obvious that the first order derivative of $v^{\kappa}$ satisfies
	\begin{equation*}
	\partial_{x}v^{\kappa}(0)>0, \quad 		\partial_{x}v^{\kappa}(1)>0.
	\end{equation*}
	Similarly, for any interval $[0,a]$, $a\in(0,1)$, the Hopf lemma holds with some given value $v^{\kappa}(a)=v_{a}$ satisfying $v_{0}<v_{a}<v_{1}$. Therefore for any internal point $x\in(0,1)$, we have $\partial_{x}v^{\kappa}(x)>0$ and \eqref{epsilon-isen-hopf} is proved.
\end{proof}

By Lemma \ref{isen-heat-prior-hopf}, $v^{\kappa}$ is a strictly increasing function of $x$, so the inverse function exists, which will be denoted by $x=X_{\kappa}(v)$ for $v\in (v_{0},v_{1})$. 
Then, $X_{\kappa}(v)$ satisfies the following problem
\begin{equation}\label{inverse}
	\left\{
	\begin{aligned}
		 & \frac{dX_{\kappa}(v)}{dv}=  \frac{\kappa}{F(v;\alpha_{\kappa})} , \\
		 & X_{\kappa}(v_{0})=0, \quad  X_{\kappa}(v_{1})=1.
	\end{aligned}
	\right.
\end{equation}
Obviously, solving the boundary value problem (\ref{first-order}) and (\ref{second-order-b.c.}) is equivalent to solving problem (\ref{inverse}). 

It can be easily seen that the solution to \eqref{inverse} can be written as:
\begin{equation}
	X_{\kappa}(v)=\int_{v_{0}}^{v}\frac{\kappa}{F(w;\alpha_{\kappa})}dw,
\end{equation}
with the unknown constant $\alpha_{\kappa}$ being determined by
\begin{equation}\label{hb_alpha}
	1=\int_{v_{0}}^{v_{1}}\frac{\kappa}{F(v;\alpha_{\kappa})}dv.
\end{equation}
Therefore, it suffices to show the existence of the solution $ \alpha_{\kappa} $ to the equation \eqref{hb_alpha} to establish the existence of the solution to the problem \eqref{inverse}, which is the consequence of the following lemma.


\begin{lem}\label{exist-alpha-isen}
	For any $\kappa>0$, there exists a unique constant $\alpha_{\kappa}>0$ such that
	\begin{equation}
		H_{\kappa}(\alpha_{\kappa})=1,
	\end{equation}
	where
	\begin{equation}
		H_{\kappa}(\alpha)=\int_{v_{0}}^{v_{1}}\frac{\kappa}{F(w,\alpha)}dw,\qquad \alpha>0.
	\end{equation}
\end{lem}
\begin{proof}

	The equation \eqref{first-order} as well as the estimates \eqref{epsilon-isen-prior} and \eqref{epsilon-isen-hopf} yields that
	\begin{equation}\label{bg_inf_alpha}
		\alpha_{\kappa}\geq -\inf_{v_{0}<v<v_{1}} f(v) = -f(v_{*})>0.
	\end{equation}
	Then it is obvious that  $H_{\kappa}(\alpha)$ is a decreasing continuous function with respect to $\alpha\geq -f(v_{*})$. 
	Since
	\begin{equation}\label{eq2.2.1}
		\lim_{\alpha \rightarrow  +\infty}H_{\kappa}(\alpha)=0,
	\end{equation}
	it suffices to show that there exists a constant $\hat{\alpha}>0$ such that
	\begin{equation}
		H_{\kappa}(\hat{\alpha})>1.
	\end{equation}
	Let $\hat{v}=v_{*}=\gamma^{\frac{1-\gamma}{1+\gamma}}$, and
	\[
		\hat{L}(v)=
		\begin{cases}-s_{1}(v-v_{0}),\quad \ \  &
             v_{0}\leq v\leq \hat{v},                              \\
             s_{2}(v-v_{1}), \quad \ \  & \hat{v}\leq v\leq v_{1},
		\end{cases}
	\]
	where
	\begin{equation}
		s_{1}:=-\frac{f(\hat{v})}{\hat{v}-v_{0}}>0, \quad  s_{2}:=\frac{f(\hat{v})}{\hat{v}-v_{1}} >0.
	\end{equation}
	Obviously, $\hat{L}(v)\geq f(v)$. For $\alpha > -f(\hat{v})$, we have
	\begin{equation}
		\begin{aligned}
			H_{\kappa}(\alpha) & = \int_{v_{0}}^{v_{1}}\frac{\kappa}{f(v)+\alpha}dv                                                                        \\
			                   & \ge \int_{v_{0}}^{v_{1}}\frac{\kappa}{\hat{L}(v)+\alpha}dv                                                                \\
			                   & =\int_{v_{0}}^{\hat{v}}\frac{\kappa}{-s_{1}(v-v_{0})+\alpha}+\int_{\hat{v}}^{v_{1}}\frac{\kappa}{s_{2}(v-v_{1})+\alpha}dv \\
			                   & \ge \int_{\hat{v}}^{v_{1}}\frac{\kappa}{s_{2}(v-v_{1})+\alpha}dv                                                          \\
			                   & =-\frac{\kappa}{s_{2}}\ln\left(1+\frac{f(\hat{v})}{\alpha}\right).
		\end{aligned}
	\end{equation}
	Hence one can choose $\displaystyle -f(\hat{v})< \hat{\alpha}_{\kappa} < \frac{-f(\hat{v})}{1-e^{-\frac{s_{2}}{\kappa}}}$ so that $H_{\kappa}(\hat{\alpha}_{\kappa})>1$.
	By \eqref{eq2.2.1} and the monotonicity of $H_{\kappa}(\alpha)$, there exists a unique constant $\alpha_{\kappa} > \hat{\alpha}_{\kappa}$ such that $H_{\kappa}(\alpha_{\kappa})=1$, which completes the proof.
\end{proof}

By Lemma \ref{exist-alpha-isen}, one immediately obtains the following lemma.
\begin{lem}\label{lem:hb-exist-solu}
	Let $\kappa>0$. There exists a unique solution $(v^{\kappa},\alpha_{\kappa})$, with $ \alpha_{\kappa} > -f(v_{*}) $, satisfying the equation (\ref{first-order}) and the boundary conditions (\ref{second-order-b.c.}).
	
\end{lem}


\subsubsection{Asymptotic behaviors of $(v^{\kappa},\alpha_{\kappa})$ as $ \kappa\sTo 0+ $}\label{Section 3.1.3}


Now we continue to investigate the asymptotic behaviors of the solutions $(v^{\kappa},\alpha_{\kappa})$ to the problem (\ref{first-order}) and  (\ref{second-order-b.c.}) as the parameter $ \kappa $ vanishes, checking that whether they will converge to a shock solution among \eqref{shock-solution_hb} or not.

By the estimate \eqref{epsilon-isen-prior}, one obtains that
\begin{equation}
	T.V.v^{\kappa}=v_{1}-v_{0},
\end{equation}
where $T.V.v^{\kappa}$ stands for the total variation of $v^{\kappa}$. 
Then by Helly's theorem, there exists a subsequence $\{\kappa_{n}\}$ and a function $v^{0}$ of bounded variation, such that $\kappa_{n}\sTo0$ as $ n\sTo\infty $, and for a.e. $ x\in[0,1] $,
\begin{equation}
	v^{\kappa_{n}}(x)\to v^{0}(x),\qquad \text{as} \quad n \to \infty.
\end{equation}
Then, one needs to check whether the limit function $v^{0}$ is a transonic shock solution among \eqref{shock-solution_hb} and whether it is the unique limit of $v^{\kappa}$. 
Following the ideas in \cite{FangZhao2021CPAA}, the argument will be divided into the following three steps.
\begin{enumerate}
	\item  Calculate the limit $\displaystyle \alpha_0= \lim_{\kappa \to 0}\alpha_{\kappa}$.
	\item  Prove that $v^0$ satisfies 
	\begin{equation*}
	F(v^{0};\alpha_{0})=0
	\end{equation*}
	 in the weak sense \eqref{isen-weak-sol}. This shows that what values $v^0$ can take.
	 
	 If $v^{0}$ only takes values in the set $\{v_{0},v_{1}\}$, namely there exists a subset $A\subset [0,1]$ and $A^{c}=[0,1]\backslash A$, such that
	 \begin{equation*}
	 	v^{0}(x)=\left\{
	 	\begin{aligned}
	 		 &v_{0}, \quad x\in A, \\
	 		 &v_{1},  \quad x\in A^{c} .
	 	\end{aligned}
	 	\right.
	 \end{equation*}
 	then we continue the next step (iii) to identify which one among \eqref{shock-solution_hb} the function $v^{0}$ can be. 
 	
 	Otherwise, the limit function $v^{0}$ cannot be any shock solution among \eqref{shock-solution_hb}. Therefore, the solutions of the problem \eqref{first-order} and  (\ref{second-order-b.c.}) cannot converge to any shock solution among \eqref{shock-solution_heat} as the parameter $ \kappa $ vanishes.
	\item  Calculate the limit of $X_\kappa$ as $ \kappa\sTo0 $ to check whether the uniqueness of $v^0$ is valid. 
\end{enumerate}

It finally turns out that it is not true that $v^{0}$ only takes values in the set $\{v_{0},v_{1}\}$.
Hence, we only carry out the first two steps, which is sufficient to show that the solutions of the problem \eqref{first-order} and  (\ref{second-order-b.c.}) will not converge to any shock solution among \eqref{shock-solution_heat} as $ \kappa\sTo0 $.

For (i), we have the following lemma:
\begin{lem}
	It holds that
	\begin{equation}\label{limit-alpha-isen-heat}
		\lim_{\kappa \to 0^{+}}\alpha_{\kappa}=-f(v_{*}),
	\end{equation}
	where
	$v_{*}=\gamma^{\frac{1-\gamma}{1+\gamma}}$.
\end{lem}
\begin{proof}
	On one hand, from the proof of lemma \ref{exist-alpha-isen}, there exists a lower bound for $\alpha_{\kappa}$  satisfying $\alpha_{\kappa}>\hat{\alpha}_{\kappa}$, where
	\begin{equation*}
		\displaystyle -f(v_{*})< \hat{\alpha}_{\kappa} < \frac{-f(v_{*})}{1-e^{-\frac{s_{2}}{\kappa}}}.
	\end{equation*}
	On the other hand, we claim that there also exists an upper bound $\tilde{\alpha}_{\kappa}$ for $\alpha_{\kappa}$.	
	
	Define a function
	\begin{equation*}
		\tilde{L}(v)=f(v_{*})\leq f(v),
	\end{equation*}
	then
		\begin{align*}
			H_{\kappa}(\alpha) & =\int_{v_{0}}^{v_{1}}\frac{\kappa}{f(v)+\alpha}dv        \\
			                   & \leq \int_{v_{0}}^{v_{1}}\frac{\kappa}{f(v_{*})+\alpha}dv \\
			                   & =\frac{\kappa(v_{1}-v_{0})}{f(v_{*})+\alpha}              \\
			                   & :=\tilde{H}_{\kappa}(\alpha).
		\end{align*}
	Let, with sufficiently small $\kappa>0$, 
	\begin{equation*}
\tilde{\alpha}_{\kappa}:=\kappa(v_{1}-v_{0})-f(v_{*})>0,
	\end{equation*}
 such that 
 \begin{equation*}
\tilde{H}_{\kappa}(\tilde{\alpha}_{\kappa})=1.
 \end{equation*}
Then one obtains 
\begin{equation*}
H_{\kappa}(\tilde{\alpha}_{\kappa})<\tilde{H}_{\kappa}(\tilde{\alpha}_{\kappa})=1=H_{\kappa}(\alpha_{\kappa}).
\end{equation*}
Since $H_{\kappa}$ is decreasing with respective to $\alpha\in(-f(v_{*}),\infty)$, it holds that
\begin{equation*}
-f(v_{*})<\alpha_{\kappa}<\tilde{\alpha}_{\kappa}.
\end{equation*}
	Hence
	\begin{equation*}
		\hat{\alpha}_{\kappa}<\alpha_{\kappa}<\tilde{\alpha}_{\kappa},
	\end{equation*}
	let $\kappa$ goes to $0$, we have
	\begin{equation*}
		\lim_{\kappa\to 0}\alpha_{\kappa}=-f(v_{*})>0,
	\end{equation*}
	which complete the proof.
\end{proof}
For (ii), the following lemma holds:
\begin{lem}\label{lem:weak solution of 3.1}
The limit function $v^{0}$ is a weak solution to limit problem \eqref{weak-sol1} in the sense that for any test function $\phi(x)\in C^{\infty}[0,1]$, it holds that
\begin{equation}
	\int_{0}^{1}F(v^{0}(x),\alpha_{0})\phi(x)dx=0,
\end{equation}
where $ \alpha_{0} = -f(v_{*}) > 0$. 
\end{lem}
\begin{proof}
For any test function $\phi(x) \in C^{\infty}[0,1]$, since $v^{\kappa}\in [v_{0},v_{1}]$, then 
\begin{align*}
\int_{0}^{1}\kappa_{n} \partial_{x}v^{\kappa_{n}}\phi dx&=\kappa_{n} v^{\kappa_{n}}\phi|_{x=0}^{x=1}-\int_{0}^{1}\kappa_{n} v^{\kappa_{n}}\partial_{x}\phi dx\\
&\longrightarrow 0, \quad as\quad \kappa_{n} \longrightarrow 0.
\end{align*}
Moreover, 
\begin{equation*}
\lim_{\kappa_{n}\longrightarrow0}\int_{0}^{1}F(v^{\kappa_{n}};\alpha_{\kappa})\phi(x)dx=\int_{0}^{1}F(v^{0};\alpha_{0})\phi(x)dx.
\end{equation*}
Then the lemma follows by \eqref{first-order}.
\end{proof}

By lemma \ref{lem:weak solution of 3.1}, for a.e. $x \in (0,1)$, $v^0 (x)$ satisfies 
\begin{equation*}
	F(v^0(x),\alpha_0)=0,
\end{equation*}
or it can be rewritten as
\begin{equation*}\label{f(v)=f(v*)}
	f(v^0(x))=-\alpha_0 = f(v_*).
\end{equation*}
Hence the limit function $v^0(x)$ can only take the value $v_*$ for a.e $x \in (0,1)$. 
Therefore $v^\kappa$ can not converge to any shock solution of the form \eqref{shock-solution_hb}. 
Then Theorem \ref{thm: isen-heat} follows from this fact as well as Lemma \ref{lem:hb-exist-solu}.

\subsection{The non-isentropic polytropic gases}\label{Section 3.2}

In this subsection, we continue to investigate the case of polytropic gases and deal with the problem [{\bf HP}].
Following the ideas in \cite{FangZhao2021CPAA}, the boundary value problem \eqref{non-heat-eq} and \eqref{non-heat-bd} will be reformulated as a problem for an ordinary equation of first order with an unknown parameter, which is studied carefully to lead us to Theorem \ref{thm: non-heat}.

\subsubsection{Reformulation of the problem [{\bf HP}]}


Let $ U=(u,\rho,p) $ be a solution to the system \eqref{non-heat-eq}. By the assumption \eqref{rho-velocity}, it holds that
\begin{equation}\label{density-u-hp}
\rho=\frac{1}{u}.
\end{equation}
Moreover, by the second equation of $(\ref{non-heat-eq})$, one has, with $A\defs\rho_{0}q_{0}^{2}+p_{0}$ a constant,
\begin{equation}\label{22}
	\rho u^{2}+p\equiv A.
\end{equation}
Then 
\[
	e = \frac{1}{\gamma-1}\frac{p}{\rho} = \frac{1}{\gamma-1}p(A-p),
\]
and the Bernoulli function 
\begin{equation}
	\Phi=\frac{1}{2}u^{2}+\frac{\gamma}{\gamma-1}\frac{p}{\rho}=\frac{1}{2}(A-p)^{2}+\frac{\gamma}{\gamma-1}p(A-p)\defs\Phi(p).
\end{equation}
Thus, the third equation of $(\ref{non-heat-eq})$ can be reformulated as, with $\displaystyle \kappa:=\frac{\epsilon}{\gamma-1}>0$,
\begin{equation}\label{non-2th-p}
	\begin{aligned}
		 & \partial_{x}(\frac{1}{2}(A-p)^{2}+\frac{\gamma}{\gamma-1}p(A-p))                            \\
		 & =\frac{\epsilon}{\gamma-1}\partial_{x}((A-2p)\partial_{x}p)                                 \\
		 & =\frac{\epsilon}{\gamma-1} (A-2p)\partial_{xx}p-2\frac{\epsilon}{\gamma-1}(\partial_{x}p)^2 \\
		 & =\kappa (A-2p)\partial_{xx}p-2\kappa(\partial_{x}p)^2.
	\end{aligned}
\end{equation}
Obviously, the boundary conditions \eqref{non-heat-bd} can be simplified as
\begin{equation}\label{hp-bd}
	p(0) = p_{0},\qquad p(1) = p_{1}.
\end{equation}
Note that, by the Rankine-Hugoniot conditions \eqref{R-H}, it holds that
\begin{equation}\label{hp-phi-bd}
	\Phi(p_{0}) = \Phi(p_{1}).
\end{equation}
Hence, the BVP \eqref{non-heat-eq} and \eqref{non-heat-bd} is reduced to a boundary value problem of a nonlinear ordinary differential equation \eqref{non-2th-p} of second order with the boundary value conditions \eqref{hp-bd} satisfying \eqref{hp-phi-bd}.

Let $f(p)\defs \Phi(p)-\Phi(p_{0})$. Then it is a smooth, strictly concave function defined in $(0,\infty)$, and satisfies
\begin{equation}
f(p_{0})=f(p_{1})=0.
\end{equation} 
Thus, it holds that
\begin{equation}\label{sign-f-hp}
f(p)>0, \qquad \text{for} \quad \text{any} \quad p \in (p_{0},p_{1}).
\end{equation}
Integrating \eqref{non-2th-p} from 0 to $x$, one obtains that the solution $ p=p^{\kappa}(x) $  to the BVP \eqref{non-2th-p} and \eqref{hp-bd} satisfies
\begin{equation}\label{non-first}
\kappa (A-2p^{\kappa})\partial_{x}p^{\kappa}=F(p^{\kappa},\alpha_{\kappa})
\end{equation}
where
\begin{equation}
F(p,\alpha):=f(p)+\alpha,
\end{equation}
and $\alpha_{\kappa}:=\kappa (A-2p_{0})\partial_{x}p^{\kappa}(0)$ is an unknown constant to be determined together with $ p^{\kappa} $ under the boundary conditions \eqref{hp-bd}. 
Therefore, the BVP \eqref{non-2th-p} and \eqref{hp-bd} is further reduced to the boundary value problem of the first order ordinary differential equation \eqref{non-first} and the boundary conditions \eqref{hp-bd} for the unknowns $ (p^{\kappa};\ \alpha_{\kappa}) $.


Moreover, the equation (\ref{non-first}) formally converges to
\begin{equation}\label{weak-sol1-hp}
F(p^{0}(x),\alpha_{0})=0
\end{equation}
as $\kappa \to 0$. 
The weak solution $(p^{\kappa},\alpha_{\kappa})$ to \eqref{weak-sol1-hp} is defined in the sense that, for any test functions $\phi(x)\in C^{\infty}[0,1]$, it holds that
\begin{equation}\label{non-weak-sol-hp}
\int_{0}^{1}F(p^{0}(x),\alpha_{0})\phi(x)dx = 0.
\end{equation}
Obviously, for any $ 0 < x_{s} <1 $, the weak solutions $(p^{0}(x;x_{s}), \alpha_{0})$ associating with the shock solutions \eqref{shock-solution} satisfy that $ \alpha_{0} = 0 $ and
\begin{equation}\label{shock-solution_hp}
p^{0}(x;x_{s}):=
\left\{
\begin{aligned}
p_{0}, & \quad 0\leq x < x_{s},  \\
p_{1}, & \quad x_{s} < x \leq 1.
\end{aligned}
\right.
\end{equation}

Thus, the singular limit problem [{\bf HP}] is reformulated as the following problem.

\vskip 5px

\underline{Problem [{\bf HP-R}]}:

\begin{quotation}
	
	Let $\kappa>0$. Try to find a solution $(p^{\kappa},\alpha_{\kappa})$ satisfying the equation \eqref{non-first} and the boundary conditions \eqref{hp-bd}.
	
	Furthermore, find shock solutions, among weak solutions $(p^{0}(x;x_{s}), \alpha_{0})$ in \eqref{shock-solution_hp} to the system \eqref{weak-sol1-hp}, which could be the limit of $(p^{\kappa},\alpha_{\kappa})$ ( or its subsequences ) as $\kappa\sTo 0$. \qed
\end{quotation}

\vskip 5px

Therefore the aim of this subsection is to solve the reformulated problem [{\bf HP-R}]. 
Analogous to the argument for the case of barotropic gases, we follow the techniques in \cite{FangZhao2021CPAA} to deal with the problem  [{\bf HP-R}].
However, there arises a new difficulty since the coefficient of the principle part of \eqref{non-first} may vanish, which will lead to degeneracy of the equation.
In this paper, certain sufficient conditions will be proposed to avoid this degeneracy, under which the existence of the solutions $(p^{\kappa},\alpha_{\kappa})$ for any $\kappa>0$ will be established. 
It is interesting to further observe that the solutions $(p^{\kappa},\alpha_{\kappa})$ do converge to one of the shock solutions in \eqref{shock-solution_hp} as $\kappa\sTo 0$, which is different from the asymptotic behaviors for the case of barotropic gases.


\subsubsection{Existence of the solution $(p^{\kappa},\alpha_{\kappa})$ to \eqref{non-first} and \eqref{hp-bd}}

To employ the ideas and techniques in \cite{FangZhao2021CPAA}, one needs that $p^{\kappa}$ is strictly monotone such that it is invertible. The following lemma presents a sufficient condition for the monotonicity.
\begin{lem}\label{hp-p-monotone}
	Suppose  that either (HP1) or (HP2) in Theorem \ref{thm: non-heat} holds.
	Then if $p^{\kappa}\in C^{2}(0,1)\cap C^{1}[0,1]$ is a solution to the boundary value problem \eqref{non-2th-p} and \eqref{hp-bd}, 
	it holds that
	\begin{equation}\label{non-prior}
		p_{0}< p^{\kappa}< p_{1}, \qquad for \quad any \quad x \in (0,1).
	\end{equation}
	Moreover, one has
	\begin{equation}\label{non-dao}
		\partial_{x}p^{\kappa}(0)>0,\qquad \partial_{x}p^{\kappa}(1)>0.
	\end{equation}
\end{lem}
\begin{proof}
	Applying lemma \ref{lemma2} to \eqref{non-first}, one obtains the \emph{a prior} estimate \eqref{non-prior} of $p^{\kappa}$.
	As for (\ref{non-dao}), it comes from the Hopf lemma. Indeed, by condition (\ref{Mach}), it holds that
		\begin{align*}
			A-2p^{\kappa}> & A-2p_{1}=\rho_{1}q_{1}^{2}-p_{1}                                                                              \\
			      & =\frac{q_{0}^{2}}{c_{*}^{2}}\rho_{0}\cdot (\frac{c_{*}^{2}}{q_{0}})^2-p_{0}((1+\nu^{2})M_{0}^{2}-\nu^{2})     \\
			      & =p_{0}\left(\gamma \nu^{2}M_{0}^{2}+\frac{2\gamma}{\gamma+1}+\nu^{2}-\frac{2\gamma}{\gamma+1}M_{0}^{2}\right) \\
			      & =p_{0}\left(\frac{\gamma^{2}-3\gamma}{\gamma+1}M_{0}^{2}+\frac{3\gamma-1}{\gamma+1}\right)                    \\
			      & >0
		\end{align*}
	where $\displaystyle \nu^{2}=\frac{\gamma-1}{\gamma+1}$, $\displaystyle M_{0}^{2}=\frac{\rho_{0}q_{0}^{2}}{\gamma p_{0}}$, and
	$$\displaystyle c_{*}^{2}=\nu^{2}(q_{0}^{2}+\frac{2\gamma}{\gamma-1}\frac{p_{0}}{\rho_{0}}).$$ 
	This inequality implies that the coefficient $\kappa (A-2p)$ before the second order term $\partial_{xx}p$ is strictly positive such that the degeneracy will not occur.
	Then the Hopf lemma can be applied to the second order equation \eqref{non-2th-p} to obtain \eqref{non-dao}.
\end{proof}
\begin{rem}
	Lemma \ref{hp-p-monotone} yields that
	\begin{equation}
		\alpha_{\kappa} = \kappa (A-2p_{0})\partial_{x}p^{\kappa}(0) > 0,
	\end{equation}
	and then, by virtue of \eqref{sign-f-hp} and the estimate \eqref{non-prior}, it holds that, for any $ x\in [0,1] $,
	\begin{equation}
		\kappa (A-2p^{\kappa})\partial_{x}p^{\kappa} = F(p^{\kappa}(x),\alpha_{\kappa}) = f(p^{\kappa}(x))+\alpha_{\kappa} > 0.
	\end{equation}
	Since $A-2p^{\kappa}>0$, it holds that $p^{\kappa}$ is a strictly increasing function of $x$ such that the inverse function $x=X_{\kappa}(p)$ exists, which satisfies the following boundary value condition problem:
	\begin{equation}\label{non-ODE}
	\left\{
	\begin{aligned}
	& \frac{dX_{\kappa}(p)}{dp}=  \frac{\kappa(A-2p)}{F(p;\alpha_{\kappa})}, \\
	& X_{\kappa}(p_{0})=0, \quad  X_{\kappa}(p_{1})=1.
	\end{aligned}
	\right.
	\end{equation}
	Obviously, solving the boundary value problem \eqref{non-first} and \eqref{hp-bd} is equivalent to solving the problem \eqref{non-ODE}.
\end{rem}


It can be easily seen that the solution $X_{\kappa}(p)$ to the problem \eqref{non-ODE} can be written as
\begin{equation}\label{hp-X}
	X_{\kappa}(p) = \int_{p_{0}}^{p}\frac{\kappa(A-2t)}{F(t;\alpha_{\kappa})}dt,
\end{equation}
with the unknown constant $\alpha_{\kappa}>0$ being determined by:
\begin{equation}\label{hp-alpha}
	1=\int_{p_{0}}^{p_{1}}\frac{\kappa(A-2p)}{F(p;\alpha_{\kappa})}dp.
\end{equation}
Therefore, it suffices to show the existence of the solution $ \alpha_{\kappa} $ to the equation \eqref{hp-alpha} to establish the existence of the solution to the problem \eqref{non-ODE}, which is the consequence of the following lemma.
\begin{lem}\label{non-alpha-exist}
	For any $\kappa>0$, there exists a unique constant $\alpha_{\kappa}>0$ such that
	\begin{equation}
		H_{\kappa}(\alpha_{\kappa})=1,
	\end{equation}
	where, for $ \alpha>0 $,
	\begin{equation}
		H_{\kappa}(\alpha)=\int_{p_{0}}^{p_{1}}\frac{\kappa(A-2p)}{F(p,\alpha)}dp.
	\end{equation}
\end{lem}
\begin{proof}
	Let
	\begin{equation}
		H_{\kappa}(\alpha):=\int_{p_{0}}^{p_{1}}\frac{\kappa(A-2p)}{F(p;\alpha)}dp=\int_{p_{0}}^{p_{1}}\frac{\kappa(A-2p)}{f(p)+\alpha}dp.
	\end{equation}
	Obviously, $H_{\kappa}(\alpha)$ is a continuous strictly decreasing function with respect to $\alpha \in (0,+\infty)$. Furthermore,
	\begin{equation}
		\lim_{\alpha\to +\infty}H_{\kappa}(\alpha)=0.
	\end{equation}
	Therefore, it suffices to prove the following inequality holds:
	\begin{equation}\label{more_than1_2}
		\lim_{\alpha \to 0}H_{\kappa}(\alpha)>1.
	\end{equation}
	Since $0<f(p)<f'(p_{0})(p-p_{0})$, it holds that
		\begin{align}
			&\int_{p_{0}}^{p_{1}}\frac{\kappa(A-2p)}{f(p)+\alpha}dp\notag \\
			& >\int_{p_{0}}^{p_{1}}\frac{\kappa A}{f'(p_{0})(p-p_{0})+\alpha}dp-\int_{p_{0}}^{p_{1}}\frac{2\kappa p}{f'(p_{0})(p-p_{0})+\alpha}dp                                    \notag                  \\
			& =\frac{\kappa A}{f'(p_{0})}\ln\left(1+\frac{f'(p_{0})(p_{1}-p_{0})}{\alpha}\right)                                                                                               \notag         \\
			& \quad - \left [ \frac{2\kappa(p_{1}-p_{0})}{f'(p_{0})}+\frac{2\kappa}{f'(p_{0})}\left(p_{0}-\frac{\alpha}{f'(p_{0})}\ln\left(1+\frac{f'(p_{0})(p_{1}-p_{0})}{\alpha}\right)\right)\right]\notag \\
			& =-\frac{2\kappa(p_{1}-p_{0})}{f'(p_{0})}+\frac{\kappa}{f'(p_{0})}\left(A-2p_{0}+\frac{2\alpha}{f'(p_{0})}\right)\ln\left(1+\frac{f'(p_{0})(p_{1}-p_{0})}{\alpha}\right).
		\end{align}
	The last term 
	$$\displaystyle -\frac{2\kappa(p_{1}-p_{0})}{f'(p_{0})}+\frac{\kappa}{f'(p_{0})}(A-2p_{0}+\frac{2\alpha}{f'(p_{0})})\ln(1+\frac{f'(p_{0})(p_{1}-p_{0})}{\alpha})$$ 
	tends to $\infty$ as $\alpha\to 0.$ Therefore \eqref{more_than1_2} holds and the proof is completed.
\end{proof}

By Lemma \ref{non-alpha-exist}, one immediately obtains the following lemma.
\begin{lem}\label{lem:hp-exist-solu}
	Let $\kappa>0$. There exists a unique solution $(p^{\kappa},\alpha_{\kappa})$, with $ \alpha_{\kappa} > 0 $, satisfying the equation \eqref{non-first} and the boundary conditions \eqref{hp-bd}.
	
\end{lem}

\subsubsection{Convergence of $(p^{\kappa},\alpha_{\kappa})$ as $ \kappa\sTo0+ $}

Now we continue to investigate the asymptotic behaviors of the solutions $(p^{\kappa},\alpha_{\kappa})$ to the problem \eqref{non-first} and \eqref{hp-bd} as the parameter $ \kappa $ vanishes, checking that whether they will converge to a shock solution among \eqref{shock-solution_hp} or not.

By the estimate \eqref{non-prior} and the monotonicity of $ p^{\kappa} $, one obtains that
\begin{equation}
T.V.p^{\kappa}=p_{1}-p_{0}.
\end{equation}
Then by Helly's theorem, there exists a subsequence $\{p^{\kappa_{n}}\}$ and a function $p^{0}$ of bounded variation, such that $\kappa_{n}\sTo0$ as $ n\sTo\infty $, and for a.e. $ x\in[0,1] $,
\begin{equation}
p^{\kappa_{n}}(x)\to p^{0}(x),\qquad \text{as} \quad n \to \infty.
\end{equation}
Then, one needs to check whether the limit function $p^{0}$ is a transonic shock solution among \eqref{shock-solution_hp} and whether it is the unique limit of $p^{\kappa}$. 
To this aim, we are going to carry out the three steps listed in the subsection \ref{Section 3.1.3}, which finally shows that, different from the case of barotropic gases, the solutions $(p^{\kappa},\alpha_{\kappa})$ to the problem \eqref{non-first} and \eqref{hp-bd} do converge to a shock solution $(p^{0},0)$ among \eqref{shock-solution_hp} as $ \kappa\sTo0+ $.


For the first step (i) listed in the subsection \ref{Section 3.1.3}, we have the following lemma
\begin{lem}\label{limit0}
	It holds that\\
	$${\lim_{\kappa\to 0^{+}}\alpha_{\kappa}=0}.$$
\end{lem}
\begin{proof}
	Let $\hat{p}=\frac{p_{0}+p_{1}}{2}$, define:
	\[
		\hat{L}(p)=
		\begin{cases}s(p-p_{0}),\quad \ \    &
             p_{0}\leq p\leq \hat{p},                           \\
             -s(p-p_{1}), \quad \ \  & \hat{p}\leq p\leq p_{1},
		\end{cases}
	\]
	where
	\begin{equation}
		s:=\frac{f(\hat{p})}{\hat{p}-p_{0}}=-\frac{f(\hat{p})}{\hat{p}-p_{1}}>0
	\end{equation}
	Obviously, $0\leq\hat{L}(p)\leq f(p)$. Then we have
		\begin{align}
			H_{\kappa}(\alpha) & = \int_{p_{0}}^{p_{1}}\frac{\kappa(A-2p)}{f(p)+\alpha}dp  \notag                                     \\
			                   & \leq \int_{p_{0}}^{p_{1}}\frac{\kappa A}{\hat{L}(p)+\alpha}dp    \notag                               \\
			                   & =\left(\int_{p_{0}}^{\hat{p}}+\int_{\hat{p}}^{p_{1}}\right)\frac{\kappa A}{\hat{L}(p)+\alpha}dp \notag\\
			                   & =\frac{2\kappa A}{s}\ln(1+\frac{s}{2\alpha}(p_{1}-p_{0}))                     \notag                   \\
			                   & =:\hat{H}_{\kappa}(\alpha)
		\end{align}
	Let $\hat{H}_{\kappa}(\alpha)=1$, then we have
	\begin{equation}
		\alpha=\frac{s(p_{1}-p_{0})}{2(e^{\frac{s}{2\kappa A}}-1)}=:\hat{\alpha}_{\kappa}
	\end{equation}
	Obviously, $$\lim_{\kappa\to 0}\hat{\alpha}_{\kappa}=0.$$ Then from $0<\alpha_{\kappa}<\hat{\alpha}_{\kappa}$, one obtains the lemma.
\end{proof}

For the step (ii), the following lemma holds:
\begin{lem}
	The limit function $p^{0}$ is a  weak solution to problem (\ref{weak-sol1-hp}), i.e. for any test function $\phi(x)\in C^{\infty}[0,1]$, \eqref{non-weak-sol-hp} holds.
\end{lem}
\begin{proof}
	For any test function $\phi(x)\in C^{\infty}[0,1]$, it holds that
		\begin{align*}
	&\quad \kappa_{n}\int_{0}^{1}(A-2p^{\kappa_{n}})\partial_{x}p^{\kappa_{n}}\phi(x)dx \\
	&=\kappa_{n} A p^{\kappa_{n}}\phi(x)|_{x=0}^{x=1}  -\kappa_{n}\int_{0}^{1}(A-2p^{\kappa_{n}})p^{\kappa_{n}}\partial_{x}\phi(x) dx \\
	&\quad -\kappa_{n}(p^{\kappa_{n}})^{2}\phi(x)|_{x=0}^{x=1}  +\kappa_{n}\int_{0}^{1}(p^{\kappa_{n}})^{2}\partial_{x}\phi(x)dx.
		\end{align*}
	Since $p_0<p^{\kappa_n}<p_1$, as $n \to \infty$, we have $\kappa_{n} \to 0$ and
	\begin{align*}
		\lim_{\kappa_{n} \to 0}
		\kappa_{n}\int_{0}^{1}(A-2p^{\kappa_{n}})\partial_{x}p^{\kappa_{n}}\phi(x)dx =0.
	\end{align*} 
	Furthermore,
	\begin{equation*}
		\lim_{\kappa_{n}\to 0}\int_{0}^{1}F(p^{\kappa_{n}},\alpha_{\kappa_{n}})\cdot \phi (x)dx=\int_{0}^{1}F(p^{0},\alpha_{0})\cdot \phi(x)dx,
	\end{equation*}
	which completes the proof of the lemma.
\end{proof}


For the step (iii), we are going to compute the limit of the following function as $\kappa \to 0+$:
\begin{equation*}
	I_{\kappa}(p):=\frac{X_{\kappa}(p)}{1-X_{\kappa}(p)},\qquad p\in (p_{0},p_{1}).
\end{equation*}
We have the following lemma:
\begin{lem}\label{I}
	For any  $p\in (p_{0},p_{1})$, it holds that
	\begin{equation}\label{I0}
	\lim_{\kappa\to 0+}I_{\kappa}(p)=\frac{A-2p_{0}}{A-2p_{1}}.
	\end{equation}
\end{lem}

\begin{proof}
	Fixed $p \in\left(p_{0}, p_{1}\right)$, then for any sufficiently small $\sigma>0$ such that, with $p^{*}$ satisfies $f'\left(p^{*}\right)=0$,
	\begin{equation*}
	p_{0}+\sigma<\min \left\{p, p^{*}\right\} \leq \max \left\{p, p^{*}\right\}<p_{1}-\sigma,
	\end{equation*}
	it holds that
	\begin{equation*}
	\int_{p_{0}}^{p} \frac{A-2 s}{F\left(p, \alpha_{\kappa}\right)} d p=\left(\int_{p_{0}}^{p_{0}+\sigma}+\int_{p_{0}+\sigma}^{p}\right) \frac{A-2 s}{F\left(p, \alpha_{\kappa}\right)} d p:=J_{1,\sigma,\kappa}+J_{2,\sigma,\kappa}
	\end{equation*}
	\begin{equation*}
	\int_{p}^{p_{1}} \frac{A-2 s}{F\left(p, \alpha_{\kappa}\right)} d p=\left(\int_{p}^{p_{1}-\sigma}+\int_{p_{1}-\sigma}^{p_{1}}\right) \frac{A-2 s}{F\left(p, \alpha_{\kappa}\right)} d p:=J_{3,\sigma,\kappa}+J_{4,\sigma,\kappa},
	\end{equation*}
	where 
	\begin{equation*}
	J_{1,\sigma,\kappa}=\int_{p_{0}}^{p_{0}+\sigma} \frac{A-2 s}{F\left(s, \alpha_{\kappa}\right)} ds,
	\end{equation*}
	\begin{equation*}
	J_{2,\sigma,\kappa}=\int_{p_{0}+\sigma}^{p} \frac{A-2 s}{F\left(s, \alpha_{\kappa}\right)} d s,
	\end{equation*}
	\begin{equation*}
	J_{3,\sigma,\kappa}=\int_{p}^{p_{1}-\sigma} \frac{A-2 s}{F\left(s, \alpha_{\kappa}\right)} ds,
	\end{equation*}
	\begin{equation*}
	J_{4,\sigma,\kappa}=\int_{p_{1}-\sigma}^{p_{1}} \frac{A-2 s}{F\left(s, \alpha_{\kappa}\right)} ds.
	\end{equation*}
	
	Since $f$ is strictly concave in $\left(p_{0}, p_{1}\right)$, then for any $\xi \in\left(p_{0}, p_{0}+\sigma\right)$, we have
	\begin{equation*}
	f^{\prime}\left(p_{0}\right)>f^{\prime}(\xi)>f^{\prime}\left(p_{0}+\sigma\right),\quad \text{for $s\in (p_{0},p_{0}+\sigma)$},
	\end{equation*}
	which implies
	\begin{equation*}
	f^{\prime}\left(p_{0}\right)\left(s-p_{0}\right)>f(s)>f^{\prime}\left(p_{0}+\sigma\right)\left(s-p_{0}\right)>0 .
	\end{equation*}
	
	Therefore, it holds that
	
	\begin{align}\label{I11}
	\int_{p_{0}}^{p_{0}+\sigma} \frac{A-2 s}{f(s)+\alpha_{\kappa}} ds & \leq \int_{p_{0}}^{p_{0}+\sigma} \frac{A-2 p_{0}}{f^{\prime}\left(p_{0}+\sigma\right)\left(s-p_{0}\right)+\alpha_{\kappa}} d s \\\notag
	& =\frac{A-2 p_{0}}{f^{\prime}\left(p_{0}+\sigma\right)} \ln \frac{\sigma f^{\prime}\left(p_{0}+\sigma\right)+\alpha_{\kappa}}{\alpha_{\kappa}}
	\end{align}
	
	and
	
	\begin{align}\label{I12}
	\int_{p_{0}}^{p_{0}+\sigma} \frac{A-2 s}{f(p)+\alpha_{\kappa}} d s & \geq \int_{p_{0}}^{p_{0}+\sigma} \frac{A-2 p}{f^{\prime}\left(p_{0}\right)\left(s-p_{0}\right)+\alpha_{\kappa}} d s \\\notag
	& =\frac{A-2 p_{0}+\frac{2 \alpha_{\kappa}}{f^{\prime}\left(p_{0}\right)}}{f^{\prime}\left(p_{0}\right)} \ln \frac{\sigma f^{\prime}\left(p_{0}\right)+\alpha_{\kappa}}{\alpha_{\kappa}}-\frac{2 \sigma}{f^{\prime}\left(p_{0}\right)} \\\notag
	& >0 .
	\end{align}
	
	Moreover, as $\theta \in\left(p_{0}+\sigma, p\right)$, since $f\left(p^{*}\right) \geq f(\theta) \geq f\left(p_{0}+\sigma\right)$, it holds
	
	\begin{align}\label{I21}
	\int_{p_{0}+\sigma}^{p} \frac{A-2 \theta}{f(\theta)+\alpha_{\kappa}} d \theta &\leq \int_{p_{0}+\sigma}^{p} \frac{A}{f\left(p_{0}+\sigma\right)+\alpha_{\kappa}} d \theta\\\notag
	&=\frac{A\left(p-p_{0}-\sigma\right)}{f\left(p_{0}+\sigma\right)+\alpha_{\kappa}}
	\end{align}
	and
	\begin{align}\label{I22}
	\int_{p_{0}+\sigma}^{p} \frac{A-2 \theta}{f(\theta)+\alpha_{\kappa}} d \theta & \geq \int_{p_{0}+\sigma}^{p} \frac{A-2 \theta}{f\left(p^{*}\right)+\alpha_{\kappa}} d \theta \\\notag
	& =\frac{A\left(p-p_{0}-\sigma\right)}{f\left(p^{*}\right)+\alpha_{\kappa}}-\frac{p^{2}-\left(p_{0}+\sigma\right)^{2}}{f\left(p^{*}\right)+\alpha_{\kappa}} \\\notag
	& >0 .
	\end{align}
	
	Combining \eqref{I11}\eqref{I12}\eqref{I21} and \eqref{I22}, one obtains
	\begin{align}
	& \frac{A-2 p_{0}}{f^{\prime}\left(p_{0}\right)+\sigma} \ln \frac{\sigma f^{\prime}\left(p_{0}+\sigma\right)+\alpha_{\kappa}}{\alpha_{\kappa}}+\frac{A\left(p-p_{0}-\sigma\right)}{f\left(p_{0}+\sigma\right)+\alpha_{\kappa}} \\\notag
	& \geq J_{1,\sigma,\kappa}+J_{2,\sigma,\kappa} \\\notag
	& \geq \frac{A-2 p_{0}+\frac{2 \alpha_{\kappa}}{f^{\prime}\left(p_{0}\right)}}{f^{\prime}\left(p_{0}\right)} \ln \left(1+\frac{\sigma f^{\prime}\left(p_{0}\right)}{\alpha_{\kappa}}\right)-\frac{2 \sigma}{f^{\prime}\left(p_{0}\right)} \\\notag
	& \quad+\frac{A\left(p-p_{0}-\sigma\right)-\left(p^{2}-\left(p_{0}+\sigma\right)^{2}\right)}{f\left(p^{*}\right)+\alpha_{\kappa}} \\\notag
	& >0 .
	\end{align}
	
	Analogously, one obtains that
	
	\begin{align}
	&\frac{A-2\left(p_{1}-\sigma\right)}{f^{\prime}\left(p_{1}\right)} \ln \frac{\alpha_{\kappa}}{-\sigma f^{\prime}\left(p_{1}\right)+\alpha_{\kappa}}+\frac{A\left(p_{1}-\sigma-p\right)}{f\left(p_{1}-\sigma\right)+\alpha_{\kappa}} \\\notag
	& \geq J_{3,\sigma,\kappa}+J_{4,\sigma,\kappa} \\\notag
	& \geq  -\frac{A-2 p_{1}+\frac{2 \alpha_{\kappa}}{f^{\prime}\left(p_{1}-\sigma\right)}}{f^{\prime}\left(p_{1}-\sigma\right)} \ln \frac{-\sigma f^{\prime}\left(p_{1}-\sigma\right)+\alpha_{\kappa}}{\alpha_{\kappa}}\\\notag
	& \quad+\frac{A\left(p_{1}-\sigma-p\right)}{f\left(p^{*}\right)+\alpha_{\kappa}}-\frac{\left(p_{1}-\sigma\right)^{2}-p^{2}}{f\left(p^{*}\right)+\alpha_{\kappa}}-\frac{2 \sigma}{f^{\prime}\left(p_{1}-\sigma\right)} \\\notag
	& >0 \text {. }
	\end{align}
	
	Hence, since $f'(p_{0})=-f'(p_{1})$, applying lemma \ref{limit0}, for given $p \in\left(p_{0}, p_{1}\right)$ and sufficiently small $\sigma>0$, one obtains
	
	$$
	\begin{gathered}
	\lim _{\kappa \rightarrow 0} \sup I_{\kappa}(p)=\lim _{\kappa \rightarrow 0} \sup \frac{\int_{p_{0}}^{p} \frac{A-2 p}{F\left(p, \alpha_{\kappa}\right)}}{\int_{p}^{p_{1}} \frac{A-2 p}{F\left(p, \alpha_{\kappa}\right)}} \leq-\frac{\left(A-2 p_{0}\right) f^{\prime}\left(p_{1}-\sigma\right)}{\left(A-2 p_{1}\right) f^{\prime}\left(p_{0}+\sigma\right)} . \\
	\lim _{\kappa \rightarrow 0} \inf I_{\kappa}(p)=\lim _{\kappa \rightarrow 0} \inf \frac{\int_{p_{0}}^{p} \frac{A-2 p}{F\left(p, \alpha_{\kappa}\right)}}{\int_{p}^{p_{1}} \frac{A-2 p}{F\left(p, \alpha_{\kappa}\right)}} \geq\frac{A-2 p_{0}}{A-2\left(p_{1}-\sigma\right)} .
	\end{gathered}
	$$
	
	Let $\sigma$ goes to 0 , then
	\begin{equation*}
	\frac{A-2 p_{0}}{A-2 p_{1}} \leq \lim _{\kappa \rightarrow 0} \inf I_{\kappa}(p) \leq \lim _{\kappa \rightarrow 0} \sup I_{\kappa}(p) \leq \frac{A-2 p_{0}}{A-2 p_{1}},
	\end{equation*}
	which implies \eqref{I0} and we complete the proof of Lemma \ref{I}.
\end{proof}

Then one can easily obtain the limit of $X_{\kappa}(p)$, that is, for any $p\in(p_{0},p_{1})$,
\begin{equation}\label{hp-shock-position}
	\lim_{\kappa\to 0+}X_{\kappa}(p)=\frac{1}{2}\frac{A-2p_{0}}{A-(p_{0}+p_{1})}=\frac{(\gamma+1)}{2(\gamma-1)}\cdot\frac{\gamma M_{0}^{2}-1}{\gamma M_{0}^{2}+1}=:X_{s},
\end{equation}
which yields that
\begin{equation*}
	\lim_{\kappa\to 0+}\int_{p_{0}}^{p_{1}}|X_{\kappa}(p)-X_{s}|dp=0
\end{equation*}
This means that $\{X_{\kappa}(p)\}$ is a Cauchy sequence in $L^{1}(p_{0},p_{1})$, which implies that $p^{\kappa}(x)$ is also a Cauchy sequence in $L^{1}(0,1)$. Hence, there exists a $p_{*}^{0}(x)\in L^{1}(0,1)$ 
such that
\begin{equation}
\lim_{\kappa\to 0+}\int_{0}^{1}|p^{\kappa}(x)-p_{*}^{0}(x)|dx=0.
\end{equation}
By \eqref{hp-shock-position}, it is easy to further check that $p_{*}^{0}(x)$
has the following form, which is one of the solution within \eqref{shock-solution_hp}:
\begin{equation}\label{p^{0}_{*}}
	p_{*}^{0}(x)=
	\left\{
	\begin{aligned}
		 & p_{0},\quad 0\leq x<X_{s}, \\
		 & p_{1},\quad X_{s}<x\leq 1
	\end{aligned}
	\right.
\end{equation}
Therefore, for any limit function $p^{0}(x)$ of the convergent subsequence $\{p^{\kappa_{n}}(x)\}$ (see Figure 3.1), it holds that
\begin{equation*}
	p^{0}(x)=p^{0}_{*}(x),\qquad a.e.\quad x\in [0,1].
\end{equation*}

\begin{figure}[htbp]
	\centering
	\includegraphics[width=9.5cm,height=6.1cm]{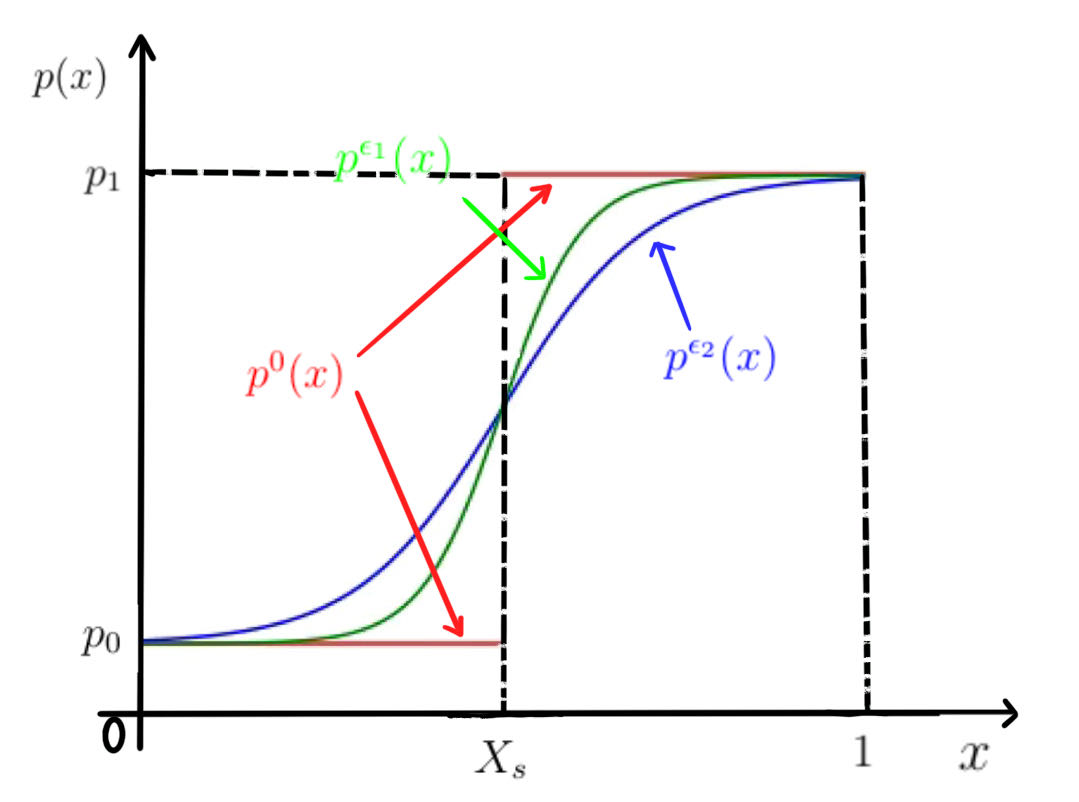}
	\caption{The pressure functions $p^{\epsilon}(x)$ for $\epsilon>0$ and their limit as $\epsilon \to 0+$.}
\end{figure}

Concluding the above argument, one obtains the following lemma:
\begin{lem}\label{lem:hp}
	Suppose  that either (HP1) or (HP2) in Theorem \ref{thm: non-heat} holds.
	Then for any $\kappa>0$, there exists a unique $C^2$ solution $(p^{\kappa},\alpha_{\kappa})$ to the boundary value problem \eqref{non-2th-p} and \eqref{hp-bd}.
	Moreover, the solution sequence $\{p^{\kappa}\}$ converges to $p_{*}^{0}$ in $  L^{1}(0,1) $ as $\kappa \to 0$, where $p_{*}^{0}$ is defined as \eqref{p^{0}_{*}}. 
\end{lem}

It is obvious that Theorem \ref{thm: non-heat} is a direct consequence of Lemma \ref{lem:hp}.

\section{Asymptotic Behaviors as the Temperature-Depending Viscosity Vanishes}

This section is devoted to analyze the asymptotic behaviors of the steady shock solutions for fluids with temperature-depending viscosity.
Again, we are going to follow the techniques in \cite{FangZhao2021CPAA} to prove Theorem \ref{thm: isen-viscous} for barotropic gases and Theorem \ref{thm: non-viscous} for polytropic gases, respectively.

\subsection{The barotropic gases}\label{Section 4.1}

In this subsection, we firstly investigate the case of barotropic gases and deal with the problem [{\bf VB}].
Following the ideas in \cite{FangZhao2021CPAA}, the boundary value problem \eqref{isen-viscous-eq} and \eqref{isen-viscous-bd} will be reformulated as a problem for an ordinary equation of first order with an unknown parameter, which is studied carefully to lead us to Theorem \ref{thm: isen-viscous}.

\subsubsection{Reformulation of the problem  [{\bf VB}]}

Let $U=U^\mu(x):=(u^{\mu}(x),\rho^{\mu}(x),p^{\mu}(x))$, with $ p^{\mu}(x) = (\rho^{\mu}(x))^{\gamma} $, be a solution to the system \eqref{isen-viscous-eq} for $\mu>0$,
then it holds that, by the assumption \eqref{rho-velocity},
\begin{equation}\label{rho-u-mu}
\rho^{\mu}=\frac{1}{u^{\mu}}.
\end{equation}
Substituting (\ref{rho-u-mu}) into the second equation in the system \eqref{isen-viscous-eq}, one obtains that
\begin{equation}\label{mu-isen-first-order}
\partial_{x} g(u^{\mu})=\mu\partial_{x}((u^{\mu})^{-\delta}\partial_{x}u^{\mu})
\end{equation}
where  
\begin{equation}\label{vb-g}
	g(u) \defs u+u^{-\gamma}.
\end{equation}
Thus, the BVP \eqref{isen-viscous-eq} and \eqref{isen-viscous-bd} is reduced to a boundary value problem of the nonlinear ordinary differential equation \eqref{mu-isen-first-order} of second order with the boundary conditions \eqref{isen-viscous-bd}.

Let $f(u):=g(u)-g(q_{0})$ be a function defined in $ (0,\infty) $, which is smooth and strictly convex. Moreover, by the Rankine-Hugoniot conditions \eqref{R-H_heat}, it holds that
\begin{equation}
	f(q_{0})=f(q_{1})=0,
\end{equation}
which implies that
\begin{equation}\label{isen-mu-sign-f}
f(u)<0, \quad \text{for} \quad \text{any} \quad u\in (q_{1},q_{0}).
\end{equation}
It can be also checked that there exists a unique $\displaystyle q_{*}:=\gamma^{\frac{1}{1+\gamma}} \in (q_{1},q_{0}) $ such that $f'(q_{*})=0$, and
\begin{equation}\label{vg_inf_f}
	\inf_{q_{1}<u<q_{0}} f(u) = f(q_{*}) <0.
\end{equation}
Integrating the equation (\ref{mu-isen-first-order}) from $0$ to $x$, one obtains 
\begin{equation}\label{isen-mu-first-order}
\mu (u^{\mu})^{-\delta}\partial_{x}u^{\mu}=F(u^{\mu},\alpha_{\mu}),
\end{equation}
where
\begin{equation}
F(u,\alpha):=f(u)+\alpha
\end{equation}
and $\alpha_{\mu}:=\mu q^{-\delta}_{0}\partial_{x}u^{\mu}(0)$ is an unknown constant which should be determined together with $u^{\mu}$ by the boundary conditions
\begin{equation}\label{insen-tem-first-bd}
	u^{\mu}(0)=q_{0}, \qquad u^{\mu}(1)=q_{1}.
\end{equation}

As $\mu \to 0$, the equation \eqref{mu-isen-first-order} formally converges to
\begin{equation}\label{weak-sol-vb}
	F(u^{0}(x),\alpha_{0})=0.
\end{equation}
The weak solution $(u^{0},\alpha_{0})$ to \eqref{weak-sol-vb} is defined in the sense that, for any test functions $\phi(x)\in C^{\infty}[0,1]$, it holds that
\begin{equation}\label{isen-weak-sol-vb}
	\int_{0}^{1}F(u^{0}(x),\alpha_{0})\phi(x)dx = 0.
\end{equation}
Obviously, for any $ 0 < x_{s} <1 $, the weak solutions $(u^{0}(x;x_{s}), \alpha_{0})$ associating with the shock solutions \eqref{shock-solution_vis} satisfy that 
\begin{equation}\label{shock-solution_vb}
	\alpha_{0} = 0,\quad\text{ and }\quad
	u^{0}(x;x_{s}):=
	\left\{
	\begin{aligned}
	q_{0}, & \quad 0\leq x < x_{s},  \\
	q_{1}, & \quad x_{s} < x \leq 1.
	\end{aligned}
	\right.
\end{equation}

Thus, the singular limit problem [{\bf VB}] is reformulated as the following problem.

\vskip 5px

\underline{Problem [{\bf VB-R}]}:

\begin{quotation}
	
	Let $\mu>0$. Try to find a solution $(u^{\mu},\alpha_{\mu})$ satisfying the equation \eqref{isen-mu-first-order} and the boundary conditions \eqref{insen-tem-first-bd}.
	
	Furthermore, find shock solutions, among weak solutions $(u^{0}(x;x_{s}), \alpha_{0})$ in \eqref{shock-solution_vb} to the system \eqref{weak-sol-vb}, which could be the limit of $(u^{\mu},\alpha_{\mu})$ ( or its subsequences ) as $\mu\sTo 0$. \qed
\end{quotation}

\vskip 5px


Therefore the aim of this subsection is to deal with the reformulated problem [{\bf VB-R}]. 
The existence of the solutions $(u^{\mu},\alpha_{\mu})$ for any $\mu>0$ will be established, and it will be proved that, as $\mu\sTo 0$, they converge to a shock solution among shock solutions $(u^{0}(x;x_{s}), \alpha_{0})$ in \eqref{shock-solution_vb} to the system \eqref{weak-sol-vb}.




\subsubsection{Existence of the solution $(u^{\mu},\alpha_{\mu})$ to \eqref{isen-mu-first-order} and \eqref{insen-tem-first-bd}}

Analogous to the argument for barotropic gases with heat conductivity in Section \ref{Section 3.1.2}, the existence of the solution $(u^{\mu},\alpha_{\mu})$ to the problem to \eqref{isen-mu-first-order} and \eqref{insen-tem-first-bd} will be established by solving the problem for the inverse function of $ u^{\mu} $, and its existence is a direct consequence of the following lemma asserting the monotonicity of $ u^{\mu} $.

\begin{lem}\label{isen-vis-prior-hopf}
	Let $\mu>0$ and $u^{\mu}\in C^{2}(0,1)\bigcap C^{1}[0,1]$ be a solution to the boundary value problem \eqref{isen-mu-first-order} with the boundary condition \eqref{insen-tem-first-bd}. Then it holds that
	\begin{equation}\label{isen-mu-prior-estimate}
		q_{1}\leq u^{\mu}\leq q_{0}, \quad  \text{for any} \quad x\in [0,1].
	\end{equation}
	
	Moreover, there holds
	\begin{equation}\label{epsilon-isen-hopf-vb}
	\partial_{x}u^{\mu}(x)<0, \quad \text{for any} \quad x\in [0,1].
	\end{equation}
\end{lem}

\begin{proof}
	Applying lemma \ref{lemma2} to \eqref{isen-mu-first-order}, one obtains that for any $\mu>0$, $u^{\mu}$ is non-increasing such that \eqref{isen-mu-prior-estimate} holds. 
	Moreover, by Hopf lemma, it holds that
\begin{equation}
	\partial_{x}u^{\mu}(0)<0,\qquad \partial_{x}u^{\mu}(1)<0,
\end{equation}
which implies
\begin{equation}\label{isen-mu-sign-alpha}
	\alpha_{\mu}<0.
\end{equation}
Then by \eqref{isen-mu-sign-f}, the equation \eqref{isen-mu-first-order}, the estimates \eqref{isen-mu-prior-estimate} and \eqref{isen-mu-sign-alpha}, one immediately obtains that \eqref{epsilon-isen-hopf-vb} holds.
\end{proof}

By \eqref{epsilon-isen-hopf-vb}, $u^\mu$ is strictly decreasing in $[0,1]$ so that its  inverse function exists, which will be denoted by $x=X_{\mu}(u)$ for $ u\in[q_1,q_0] $. Then $X_{\mu}$ satisfies the following problem:
\begin{equation}\label{isen-tem-ODE}
	\left\{
	\begin{aligned}
		& \frac{d X_{\mu}(u)}{du}=\frac{\mu}{u^{\delta}F(u,\alpha_{\mu})}, \\
		& X_{\mu}(q_{0})=0, \quad X_{\mu}(q_{1})=1.
	\end{aligned}
	\right.
\end{equation}
Obviously, solving the boundary value problem \eqref{isen-mu-first-order} and \eqref{insen-tem-first-bd} is equivalent to solving the problem \eqref{isen-tem-ODE}.

It can be easily seen that the solution to \eqref{isen-tem-ODE} can be written as
\begin{equation}\label{representation of X_mu}
	X_{\mu}(u)=-\int_{u}^{q_{0}}\frac{\mu}{v^{\delta}F(v,\alpha_{\mu})}dv
\end{equation}
with the unknown constant $\alpha_{\mu}<0$ being determined by
\begin{equation}\label{vb-alpha}
	1=-\int_{q_{1}}^{q_{0}}\frac{\mu}{u^{\delta}F(u,\alpha_{\mu})}du.
\end{equation}
Therefore, it suffices to show the existence of the solution $ \alpha_{\mu} $ to the equation \eqref{vb-alpha} to establish the existence of the solution to the problem \eqref{isen-tem-ODE}, which is the consequence of the following lemma.

\begin{lem}\label{lem:vb-alpha}
	For any $\mu>0$, there exists a unique constant $\alpha_{\mu}<0$ such that
	\begin{equation}
		H_{\mu}(\alpha_{\mu})=1,
	\end{equation}
	where
	\begin{equation}
		H_{\mu}(\alpha):=-\int_{q_{1}}^{q_{0}}\frac{\mu}{u^{\delta}F(u,\alpha)}du,\qquad \alpha<0.
	\end{equation}
\end{lem}
\begin{proof}
	Obviously, $H_\mu(\alpha)$ is a strictly increasing, continuous function and $$\displaystyle \lim_{\alpha \to -\infty}H_{\mu}(\alpha)=0.$$ 
	We claim that
	\begin{equation}\label{cc}
		\lim_{\alpha \to 0-}H_{\mu}(\alpha)=+\infty.
	\end{equation}
	Indeed, since $f$ is convex and $u$ is bounded,  we have $0>f(u)>f'(q_{1})(u-q_{1})$, which implies
	\begin{equation}
		\begin{aligned}
			H_{\mu}(\alpha) 
			&=\int_{q_{1}}^{q_{0}}\frac{\mu}{-u^{\delta}(f(u)+\alpha)}du\\
			&>\int_{q_{1}}^{q_{0}}\frac{\mu}{-q_{1}^{\delta}(f'(q_{1})(u-q_{1})+\alpha)}du \\
			& =-\frac{\mu}{q_{1}^{\delta}}\frac{1}{f'(q_{1})}\ln\frac{f'(q_{1})(q_{0}-q_{1})+\alpha}{\alpha}.
		\end{aligned}
	\end{equation}
	Therefore \eqref{cc} holds and by the continuity and the monotonicity of $H_{\mu}(\alpha)$, we complete the proof of this lemma.
\end{proof}

By Lemma \ref{lem:vb-alpha}, one immediately obtains the following lemma.
\begin{lem}\label{lem:vb-exist-solu}
	Let $\mu>0$. There exists a unique solution $(u^{\mu},\alpha_{\mu})$, with $ \alpha_{\mu} < 0 $, satisfying the equation \eqref{isen-mu-first-order} and the boundary conditions \eqref{insen-tem-first-bd}.
	
\end{lem}

\subsubsection{Convergence of $(u^{\mu},\alpha_{\mu})$ as $ \mu\sTo 0+ $}

Now we continue to investigate the asymptotic behaviors of the solutions $(u^{\mu},\alpha_{\mu})$ to the problem \eqref{isen-mu-first-order} and \eqref{insen-tem-first-bd} as the parameter $ \mu $ vanishes, checking that whether they will converge to a shock solution among \eqref{shock-solution_vb} or not.

Since the $u^{\mu}$ is strictly decreasing for any $\mu>0$, by the estimate \eqref{isen-mu-prior-estimate}, one obtains that
\begin{equation}
T.V.u^{\mu}=q_{0}-q_{1}.
\end{equation}
Then by Helly's theorem, there exists a subsequence $\{u^{\mu_n}\}$ and a function $u^{0}$ of bounded variation, such that $\mu_{n}\sTo0$ as $ n\sTo\infty $, and for a.e. $ x\in[0,1] $,
\begin{equation}
u^{\mu_n}(x)\to u^{0}(x),\qquad \text{as} \quad n \to \infty.
\end{equation}
Then, one needs to check whether the limit function $u^{0}$ is a transonic shock solution among \eqref{shock-solution_vb} and whether it is the unique limit of $u^{\mu}$. 
To this aim, we are going to carry out the following three steps similar as the ones listed in the Section \ref{Section 3.1.3}, which finally shows that the solutions $(u^{\mu},\alpha_{\mu})$ to the problem \eqref{isen-mu-first-order} and \eqref{insen-tem-first-bd} converge to a shock solution $(u^{0},0)$ among \eqref{shock-solution_vb} as $ \mu\sTo0+ $ and Theorem \ref{thm: isen-viscous} holds.

\begin{itemize}
	\item [(i)] Calculate the limit $\displaystyle \alpha_0= \lim_{\mu \to 0}\alpha_{\mu}$.
	\item [(ii)] Prove that $u^0$ satisfies the following equation in the weak sense:
	\begin{align}\label{limit-eq-vb}
		F(u,\alpha_0)=0,
	\end{align} 
	which is the formal limit of \eqref{isen-mu-first-order}. This shows that what values $u^0$ can take.
	\item [(iii)] Calculate the limit of $X_\mu$ to prove the uniqueness of $u^0$. 
\end{itemize}

For the first step (i), the following lemma shows that $\alpha_0=0$.
\begin{lem}
	It holds that $\displaystyle \lim_{\mu \to 0+} \alpha_\mu =0$.
\end{lem}
\begin{proof}
	Let $\hat{u}:=\frac{q_{0}+q_{1}}{2}$, and
	\begin{equation*}
		\hat{L}(u):=\left\{
		\begin{aligned}
			& -s(u-q_{1}),\quad & q_{1}\le u < \hat{u},  \\
			& s(u-q_{0}),\quad  & \hat{u} \le u \le q_{0},
		\end{aligned}
		\right.
	\end{equation*}
	where $\displaystyle s:=\frac{f(\hat{u})}{\hat{u}-q_{0}}=-\frac{f(\hat{u})}{\hat{u}-q_{1}}$.
	Since $f$ is strictly convex and $f(q_{0})=f(q_{1})=0$, one has that
	\begin{equation*}
		f(u)\le \hat{L}(u)\le 0, \quad \text{for any} \quad u\in[q_{1},q_{0}],
	\end{equation*}
	and $f(u)=\hat{L}(u)$ if and only if $u$ take values at $q_{0}$, $q_{1}$ and $\hat{u}$.
	Therefore, for any $\alpha<0$, we have
	\begin{equation}
		\begin{aligned}
			H_{\mu}(\alpha) & =\int_{q_{1}}^{q_{0}}\frac{\mu}{u^{\delta}(-f(u)-\alpha)}du<\int_{q_{1}}^{q_{0}}\frac{\mu}{q_{1}^{\delta}(-\hat{L}(u)-\alpha)}               \\
			& =\frac{\mu}{q_{1}^{\delta}}\left(\int_{q_{1}}^{\hat{u}}\frac{1}{s(u-q_{1})-\alpha}+\int_{\hat{u}}^{q_{0}}\frac{1}{-s(u-q_{0})-\alpha}\right) \\
			& =\frac{\mu}{s\cdot q_{1}^{\delta}}\cdot \left(\ln\frac{s(\hat{u}-q_{1})-\alpha}{-\alpha}+\ln\frac{-s(\hat{u}-q_{0})-\alpha}{-\alpha}\right)    \\
			& =\frac{2\mu}{s\cdot q_{1}^{\delta}}\cdot \ln\left(1-s(q_{0}-q_{1})\cdot\frac{1}{2\alpha}\right)                                               \\
			& =:\hat{H}(\alpha)
		\end{aligned}
	\end{equation}
	Then for sufficiently small $\mu>0$, there exists
	\begin{equation*}
		\hat{\alpha}_{\mu}:=\frac{s(q_{0}-q_{1})}{2(1-e^{\frac{s\cdot q_{1}^{\delta}}{2\mu}})}
	\end{equation*}
	such that
	\begin{equation*}
		H_{\mu}(\hat{\alpha}_{\mu})<\hat{H}_{\mu}(\hat{\alpha}_{\mu})=1=H_{\mu}(\alpha_{\mu})
	\end{equation*}
	Since $H_{\mu}$ is strictly increasing with respect to $\alpha\in (-\infty,0)$, it holds that
	\begin{equation*}
		\hat{\alpha}_{\mu}<\alpha_{\mu}<0.
	\end{equation*}
	Then we complete the proof of this lemma since
	\begin{equation*}
		\lim_{\mu \to 0+}\hat{\alpha}_{\mu}=0.
	\end{equation*}
\end{proof}


For the second step (ii), 
we prove that $u^0$ satisfies \eqref{limit-eq-vb} in the weak sense by taking the limit of \eqref{isen-mu-first-order}, which implies that $u^0$ can only take the value of $q_0$ or $q_1$ almost everywhere.

\begin{lem}
	The limit function $u^{0}$ is a weak solution to the limit problem \eqref{limit-eq-vb}, i.e. for any test functions $\phi(x) \in C^\infty[0,1]$, 
	\begin{align}
		\int_0^1 F(u^0(x),\alpha_0) \cdot \phi(x) dx =0.
	\end{align}
\end{lem}
\begin{proof}
	Recalling that $q_1 \leq u^\mu(x) \leq q_0$ for $x \in [0,1]$, one has that
	for any test function $\phi(x)\in C^{\infty}[0,1]$, if $\delta \neq 1$, it holds that
	\begin{equation*}
		\begin{aligned}
			\mu_n \int_{0}^{1} \partial_{x} u^{\mu_n} \cdot (u^{\mu_n})^{-\delta}\cdot \phi dx 
			& =\mu_n \int_{0}^{1}\frac{1}{1-\delta}\partial_{x}((u^{\mu_n})^{1-\delta})\cdot \phi dx \\
			& =\frac{\mu_n}{1-\delta}\left((u^{\mu_n})^{1-\delta}\phi|_{x=0}^{x=1} -\int_{0}^{1}(u^{\mu_n})^{1-\delta}\partial_x \phi dx \right)               \\
			& \to 0, \quad \text{as}\quad \mu_n \to 0.
		\end{aligned}
	\end{equation*}
	
	And if $\delta=1$, the same conclusion also holds if we replace $(u^{\mu_n})^{1-\delta}$ by $\ln u^{\mu_n}$.
	On the other hand,
	\begin{equation*}
		\lim_{\mu_n \to 0}\int_{0}^{1}F(u^{\mu_n},\alpha_{\mu_n})\cdot \phi dx =\int_{0}^{1}F(u^{0};\alpha_{0}) \cdot \phi dx,
	\end{equation*}
	which completes the proof of the lemma.
\end{proof}

For the third step (iii), we are going to illustrate that the solution $u^{0}$ is the unique limit by calculating the limit of $X_\mu$ as $ \mu\sTo 0+ $. Since the method is similar to Section \ref{Section 3.2}, so we give the consequence without proof. 
\begin{lem}
	Let
	\begin{equation}
	I_{\mu}(u)=\frac{X_{\mu}(u)}{1-X_{\mu}(u)}.
	\end{equation}
	Then for any $u\in(q_{1},q_{0})$, it holds that
	\begin{equation}
		\lim_{\mu\to 0}I_{\mu}(u)=-\frac{f'(q_{1})}{f'(q_{0})}(\frac{q_{1}}{q_{0}})^{\delta}.
	\end{equation}
\end{lem}
This lemma implies that

\begin{lem}
	Let
	\begin{equation}
	u^{0}_{*}(x):=
	\left\{
	\begin{aligned}
	q_{0}, & \quad 0\leq x < X_{s}  \\
	q_{1}, & \quad X_{s} < x \leq 1
	\end{aligned}
	\right.
	\end{equation}
	with
	\begin{equation} 
	X_{s}:=\left(1+(\frac{q_{0}}{q_{1}})^{\delta-\gamma-1}\frac{q_{0}^{\gamma+1}-\gamma}{\gamma-q_{1}^{\gamma+1}}\right)^{-1} = \left(1+(\frac{M_0}{M_1})^\frac{2\delta}{\gamma+1} \frac{1-{M_0}^{-2}}{M_1^{-2}-1}\right)^{-1}.
	\end{equation}	
	It holds that $u^{\mu}$ converges to $u^0_{*}$ in $L^{1}(0,1)$ as $\mu \to 0+$, namely,
	\begin{equation}
		\lim\limits_{\mu \to 0+}\int_{0}^{1}{|u^{\mu}(x)-u^{0}_{*}(x)|dx = 0}.
	\end{equation}

\end{lem}

It is obvious that Theorem \ref{thm: isen-viscous} is a consequence of the above lemma.

\subsection{The non-isentropic polytropic gases}\label{Section 4.2}

In this subsection, we continue to investigate the case of polytropic gases and deal with the problem [{\bf VP}].
Following the ideas in \cite{FangZhao2021CPAA}, the boundary value problem \eqref{non-viscous-eq} and \eqref{non-viscous-bd} will be reformulated as a problem for an ordinary equation of first order with an unknown parameter. 
It turns out that the parameter $ \delta $ in \eqref{vis_tem_function} plays an important role in the analysis.
The argument finally leads us to Theorem \ref{thm: non-viscous}.

\subsubsection{Reformulation of the problem [{\bf VP}]}

Let $U=U^{\mu}(x):=(u^{\mu}(x),p^{\mu}(x),\rho^{\mu}(x))$ be a solution to the system \eqref{non-viscous-eq} for any $ \mu>0 $.
Then it holds that, by the assumption \eqref{rho-velocity},
%
\begin{equation*}
	\rho^{\mu}=\frac{1}{u^{\mu}}.
\end{equation*}
Then \eqref{non-viscous-eq} can be simplified into
\begin{equation}\label{non-viscous-eq-simplified}
	\left\{
	\begin{aligned}
		& \partial_{x} (u^\mu + p^\mu)  =\mu \partial_{x}((T^\mu)^{\delta}\partial_{x}u^\mu), \\
		& \partial_{x} \Phi^\mu  =\mu \partial_{x}((T^\mu)^{\delta} u^\mu \partial_{x} u^\mu),
	\end{aligned}
	\right.
\end{equation}
Let
\begin{equation*}
	P_{0}:=\rho_{0} q_{0}^{2}+p_{0}, \quad
	\Phi_{0}:=\frac{1}{2}q_{0}^{2}+\frac{\gamma}{\gamma-1}\frac{p_{0}}{\rho_{0}}.
\end{equation*} 
Then integrating \eqref{non-viscous-eq-simplified} from $0$ to $x$, one has:
\begin{equation}\label{non-tem-inte}
	\left\{
	\begin{aligned}
		& u^{\mu}+p^{\mu}-P_{0}  = \mu T^{\delta}\partial_{x}u-\alpha_\mu  ,     \\
		& \Phi^{\mu}-\Phi_{0}    = \mu T^{\delta}u\partial_{x}u- \alpha_\mu q_0,
	\end{aligned}
	\right.
\end{equation}
where $\alpha_{\mu}=\mu T^{\delta}_{0}\partial_{x}u(0)$, with $T_{0}$ satisfying $p_{0}=\rho_{0}RT_{0}$, is an unknown constant which should be determined together with $u^{\mu}$ by the boundary conditions
\begin{equation}\label{nonsen-tem-first-bd}
	u^{\mu}(0)=q_{0}, \qquad u^{\mu}(1)=q_{1}.
\end{equation}

By eliminating the terms $T^{\delta}\partial_{x}u$ in $(\ref{non-tem-inte})$, one obtains
\begin{equation}\label{non-p-rep}
	p^{\mu}=\frac{\gamma-1}{2}u^{\mu}-(\gamma-1)P_{0}+(\gamma-1)\frac{\Phi_{0}}{u^{\mu}}+(\gamma-1)\alpha_{\mu}(1-\frac{q_{0}}{u^{\mu}}),
\end{equation}
Substituting \eqref{non-p-rep} into the first equation of \eqref{non-tem-inte}, one obtains
\begin{equation}
	\mu\frac{1}{R^{\delta}}(u^{\mu})^{\delta}(p^{\mu})^{\delta}\partial_{x}u^{\mu} = \frac{\gamma+1}{2}u^{\mu}-\gamma P_{0} +(\gamma-1)\frac{\Phi_{0}}{u^{\mu}}+\alpha_{\mu} (\gamma-(\gamma-1)\frac{q_{0}}{u^{\mu}})
\end{equation}
namely,
\begin{equation}\label{tem-non-first-eq}
	\begin{aligned}
		\mu(\frac{\gamma-1}{R})^{\delta}\partial_x u^{\mu} 
		& =F(u^{\mu},\alpha_{\mu}),
	\end{aligned}
\end{equation}
where $F$ is defined as
\begin{equation}
	F(u,\alpha):=\frac{f_{1}(u)+\alpha g_{1}(u)}{(f_{2}(u)+\alpha g_{2}(u))^{\delta}},
\end{equation}
with
\begin{align}
	f_{1}(u)&=\frac{\gamma+1}{2}u-\gamma P_{0} +(\gamma-1)\frac{\Phi_{0}}{u},\label{formula of f_1}\\
	g_{1}(u)&=\gamma-(\gamma-1)\frac{q_{0}}{u},\label{formula of g_1}\\
	f_{2}(u)&=\frac{1}{2}u^{2}-P_{0}u+\Phi_{0},\label{formula of f_2}\\
	g_{2}(u)&=u-q_{0},\label{formula of g_2}
\end{align}
being smooth functions defined in $(0,\infty)$. 
Obviously,  $f_{1}$ is strictly convex and, by the Rankine-Hugoniot conditions \eqref{R-H}, it satisfies
\begin{equation}
f_{1}(q_{0})=f_{1}(q_{1})=0,
\end{equation}
which implies that
\begin{equation}
f_{1}(u)<0,\quad \text{for $u\in(q_{1},q_{0})$.}
\end{equation}
Then the boundary value problem \eqref{non-viscous-eq} with boundary conditions \eqref{non-viscous-bd} is equivalent to the boundary value problem \eqref{tem-non-first-eq} with boundary conditions \eqref{nonsen-tem-first-bd}.

As $\mu \to 0$, the equation \eqref{tem-non-first-eq} formally converges to
\begin{equation}\label{weak-sol-vp}
F(u^{0}(x),\alpha_{0})=0.
\end{equation}
The weak solution $(u^{0},\alpha_{0})$ to \eqref{weak-sol-vp} is defined in the sense that, for any test functions $\phi(x)\in C^{\infty}[0,1]$, it holds that
\begin{equation}\label{non-weak-sol-vp}
\int_{0}^{1}F(u^{0}(x),\alpha_{0})\phi(x)dx = 0.
\end{equation}
Obviously, for any $ 0 < x_{s} <1 $, the weak solutions $(u^{0}(x;x_{s}), \alpha_{0})$ associating with the shock solutions \eqref{shock-solution} satisfy that 
\begin{equation}\label{shock-solution_vp}
\alpha_{0} = 0,\quad\text{ and }\quad
u^{0}(x;x_{s}):=
\left\{
\begin{aligned}
q_{0}, & \quad 0\leq x < x_{s},  \\
q_{1}, & \quad x_{s} < x \leq 1.
\end{aligned}
\right.
\end{equation}

Thus, the singular limit problem [{\bf VP}] is reformulated as the following problem.

\vskip 5px

\underline{Problem [{\bf VP-R}]}:

\begin{quotation}
	
	Let $\mu>0$. Try to find a solution $(u^{\mu},\alpha_{\mu})$ satisfying the equation \eqref{tem-non-first-eq} and the boundary conditions \eqref{nonsen-tem-first-bd}.
	
	Furthermore, find shock solutions, among weak solutions $(u^{0}(x;x_{s}), \alpha_{0})$ in \eqref{shock-solution_vp} to the system \eqref{weak-sol-vp}, which could be the limit of $(u^{\mu},\alpha_{\mu})$ ( or its subsequences ) as $\mu\sTo 0$. \qed
\end{quotation}

\vskip 5px


Therefore the aim of this subsection is to deal with the reformulated problem [{\bf VP-R}]. 
It turns out that the parameter $ \delta>0 $ plays an important role in the analysis.
The case $ 0<\delta\leq 1 $ and the case $ \delta>1 $ will behave differently, as shown in Theorem \ref{thm: non-viscous}.

\subsubsection{Analysis for the problem [{\bf VP-R}]}

The techniques to deal with the problem [{\bf VP-R}] are analogous as in Section \ref{Section 3.2}.
The key ingredient is to analyze the problem for the inverse function of $ u^{\mu} $, and it needs to show that $ u^{\mu} $ is strictly monotone such that its inverse function exists.

Applying lemma \ref{lemma2} to \eqref{tem-non-first-eq}, we have that $u^{\mu}$ is a non-increasing function and 
\begin{equation*}
q_1 \leq u^\mu(x) \leq q_0, \quad \text{for $x \in [0,1]$}.
\end{equation*}

Moreover, multiplying the first equation of \eqref{non-viscous-eq-simplified} by u and taking the difference with the second equation of \eqref{non-viscous-eq-simplified}, one obtains
\begin{equation*}
\partial_x p^\mu = \mu (\gamma-1) (T^\mu)^\delta \frac{(\partial_x u^\mu)^2}{u^\mu}
-\gamma p^\mu \frac{\partial_x u^\mu}{u^\mu}.
\end{equation*}
Substituting it into the first equation of \eqref{non-viscous-eq-simplified}, it holds that
\begin{equation*}
\mu (T^\mu)^\delta \partial_{xx} u^\mu = \partial_x u^\mu + \mu (\gamma-1) \frac{(T^\mu)^\delta}{u^\mu} (\partial_x u^\mu)^2 - \gamma \frac{p^\mu}{u^\mu} \partial_x u^\mu
- \mu \partial_x ((T^\mu)^\delta ) \partial_x u^\mu.
\end{equation*}
Therefore, by writing 
\begin{align*}
	a(x)&=\mu (T^\mu(x))^\delta, \\
	b(x)&=1 + \mu (\gamma-1) \frac{(T^\mu(x))^\delta}{u^\mu(x)} \partial_x u^\mu(x) - \gamma \frac{p^\mu(x)}{u^\mu(x)}
	- \mu \partial_x ((T^\mu(x))^\delta ),
\end{align*}
$u^\mu(x)$ satisfies the equation
\begin{align*}
	Lv=a(x) \partial_{xx} v - b(x) \partial_{x} v =0.
\end{align*}
Notice that, if $u^\mu$ is a $C^1$ solution, $a(x)>0$ and $b(x)$ are bounded functions
so that the strong maximum principle and Hopf lemma imply
\begin{equation}
\partial_{x}u^{\mu}(0)<0,\quad \partial_{x}u^{\mu}(1)<0.
\end{equation} 
The former one implies that 
\begin{equation}\label{non-mu-sign-alpha}
\alpha_{\mu}<0
\end{equation}
and then the latter one deduces that the condition \eqref{mu-gamma}, which is equivalent to $g_1(q_1)>0$, is necessary for the existence of $u^\mu$ by \eqref{tem-non-first-eq}.

Since $q_1 < u^\mu(x) < q_0$ for $x \in (0,1)$, it is easy to see that $f_1(u^{\mu})<0, f_2(u^\mu)>0, g_2(u^\mu)<0$ and $g_1(u^\mu)>0$ under condition \eqref{mu-gamma}. With $\alpha_\mu<0$, we have $F(u^\mu, \alpha_\mu)<0$ which implies that $u^\mu$ is strictly decreasing in $[0,1]$. Similar as Section \ref{Section 4.1}, there exists an inverse function which is denoted by $x=X_\mu(u)$. 
Then $X_{\mu}$ satisfies the following problem:
\begin{equation}\label{mu-non-inve}
	\left\{
	\begin{aligned}
		& \frac{d X_{\mu}(u)}{du}=\left(\frac{\gamma-1}{R}\right)^{\delta} \frac{\mu}{F(u,\alpha_{\mu})}, \\
		& X_{\mu}(q_{0})=0, \quad X_{\mu}(q_{1})=1.
	\end{aligned}
	\right.
\end{equation}
Therefore, solving problem \eqref{tem-non-first-eq}\eqref{nonsen-tem-first-bd} is equivalent to solving the problem \eqref{mu-non-inve}.

It can be easily seen that the solution $ (X_{\mu}(u);\ \alpha_{\mu}) $ to \eqref{mu-non-inve} can be written as
\begin{equation}\label{representation-vp}
	X_{\mu}(u) = -\int_{u}^{q_{0}} \left(\frac{\gamma-1}{R}\right)^{\delta} \frac{\mu}{F(v,\alpha_{\mu})} \dif v,
\end{equation}
with the unknown constant $ \alpha_{\mu} < 0 $ being determined by
\begin{equation}\label{vp-alpha}
	1 = -\int_{q_{1}}^{q_{0}} \left(\frac{\gamma-1}{R}\right)^{\delta} \frac{\mu}{F(v,\alpha_{\mu})} \dif v.
\end{equation}
Therefore, it suffices to show the existence of the solution $\alpha_\mu$ to the equation \eqref{vp-alpha} to establish the existence of the solution to the problem \eqref{mu-non-inve}.
Moreover, as $\mu\sTo 0$, the constant $\alpha_\mu$ converges to $ 0 $ is necessary such that the solution $(u^{\mu};\alpha_{\mu})$ may converge to a shock solution among \eqref{shock-solution_vp}.
The following lemma states the existence of the solution $\alpha_\mu$ to the equation \eqref{vp-alpha} and their asymptotic behaviors as $\mu\sTo 0$. 
It turns out that the value of the parameter $\delta$ plays an important role in the analysis.

\begin{lem}\label{exist-alpha_mu}
	Suppose that \eqref{mu-gamma} holds. Then
	\begin{itemize}
		\item[(i)] If $\delta \leq 1$, there exists a constant $\mu_1>0$ such that for any $\mu \in (0,\mu_1)$ there exists a unique solution $(u^{\mu};\alpha_{\mu})$ to the reduced problem  \eqref{tem-non-first-eq} with boundary condition \eqref{nonsen-tem-first-bd}. Moreover,
		\begin{equation}
			\lim_{\mu\to 0^{+}}\alpha_{\mu} = 0.
		\end{equation}
		\item[(ii)] If $\delta>1$, there exists a constant $\mu_2>0$ such that for any $\mu \in (0,\mu_2)$, there exist two solutions $(u^{\mu}_1; \alpha_{\mu 1})$ and $(u^{\mu}_2; \alpha_{\mu 2})$ to the reduced problem  \eqref{tem-non-first-eq} with boundary condition \eqref{nonsen-tem-first-bd}, satisfying
		\begin{equation}
			\lim_{\mu\to 0^{+}}\alpha_{\mu 1} = 0, \quad \lim_{\mu\to 0^{+}}\alpha_{\mu 2} = -\infty.
		\end{equation}
	\end{itemize}
\end{lem}
\begin{proof}
	\emph{Step 1:} The existence of $\alpha_{\mu}$ for sufficiently small $ \mu>0 $.\par
	By \eqref{vp-alpha}, one has
	\begin{equation}\label{mu-non-inte}
		\frac{1}{\mu}\left(\frac{R}{\gamma-1}\right)^{\delta} = -\int_{q_{1}}^{q_{0}}\frac{(f_{2}(u)+\alpha_{\mu}g_{2}(u))^{\delta}}{f_{1}(u)+\alpha_{\mu}g_{1}(u)}du.
	\end{equation}
	Define
	\begin{align*}
		J(\alpha):= - \int_{q_{1}}^{q_{0}} \frac{1}{F(u,\alpha)} du =  -\int_{q_{1}}^{q_{0}}\frac{(f_{2}(u)+\alpha g_{2}(u))^{\delta}}{f_{1}(u)+\alpha g_{1}(u)}du.
	\end{align*}
	Obviously, $ J(\alpha) $ is a continuous function for $ \alpha\in(-\infty, 0) $, and, by \eqref{mu-gamma}, it can be easily verified that
	\[ 0 \leq \inf_{\alpha < 0} J(\alpha) < +\infty.\]
	Furthermore, it is easy to see that as $\alpha \to 0^-$,
	\begin{align}\label{asym of J1}
		& J(\alpha) \to +\infty &  & \text{for any $\delta >0$},
	\end{align}
	and as $\alpha \to -\infty$,
	\begin{align}
		\label{asym of J2} & J(\alpha) \to +\infty, &  & \text{if $\delta >1$},  \\
		\label{asym of J3} & J(\alpha) \to C_0>0 ,  &  & \text{if $\delta =1$},  \\
		\label{asym of J4} & J(\alpha) \to 0,       &  & \text{if $0<\delta<1$}.
	\end{align}
	
	By the continuity of $J(\alpha)$, for any sufficiently small $\mu>0$ such that 
	$$\displaystyle \frac{1}{\mu}\left(\frac{R}{\gamma-1}\right)^{\delta}> \inf_{\alpha < 0} J(\alpha) \geq 0,$$ 
	there exists at least one $\alpha_{\mu} \in (-\infty,0 )$ such that \eqref{mu-non-inte} holds.
	That is, as $ \mu > 0 $ is sufficiently small, there exists at least one solution $\alpha_{\mu} \in (-\infty,0 )$ to the equation \eqref{vp-alpha}, which yields the existence of the solution $ (X_{\mu}(u);\ \alpha_{\mu}) $ to \eqref{mu-non-inve}.
	
	\emph{Step 2:} The uniqueness of $\alpha_{\mu}$ as $0<\delta\le 1$.\par
	It suffices to prove the monotonicity of $J(\alpha)$. Indeed,
	\begin{equation}\label{derivative of F}
		\begin{aligned}
			\frac{\partial}{\partial \alpha}\frac{1}{F(u;\alpha)}
			& = \frac{(f_{2}(u)+\alpha g_{2}(u))^{\delta-1}}{(f_{1}(u) +\alpha g_{1}(u))^2}
			\big((g_{2}(u)f_{1}(u)-g_{1}(u)f_{2}(u))                                            \\
			& \qquad +(\delta-1)g_{2}(u)(f_{1}(u)+\alpha g_{1}(u))\big)   
		\end{aligned}
	\end{equation}
	Thus, as $ 0<\delta\le 1 $, it holds that
	\begin{equation}
		\frac{\partial}{\partial \alpha}\frac{1}{F(u;\alpha)} \leq \frac{(f_{2}(u)+\alpha g_{2}(u))^{\delta-1}}{(f_{1}(u) +\alpha g_{1}(u))^2}
		(g_{2}(u)f_{1}(u)-g_{1}(u)f_{2}(u)).
	\end{equation}
	Notice that, by the Rankine-Hugoniot conditions \eqref{R-H}, it can be easily verified that $P_0 = \frac{\gamma+1}{2 \gamma} (q_0 + q_1)$
	and $ \Phi_{0} = \frac{\gamma +1}{2(\gamma-1)}q_0 q_1$. Then one has
	\begin{align*}
		g_{2}(u)f_{1}(u)-g_{1}(u)f_{2}(u)
		& = \frac{1}{2} u^2 - q_0 u + P_0 q_0 - \Phi_{0}                                                 \\
		& = \frac{1}{2} u^2 - q_0 u +\frac{\gamma+1}{2 \gamma} q_0^2
		- \frac{\gamma+1}{2\gamma(\gamma-1)} q_0 q_1                                                      \\
		& \leq \frac{1}{2} q_1^2 - q_0 q_1 +\frac{\gamma+1}{2 \gamma} q_0^2
		- \frac{\gamma+1}{2\gamma(\gamma-1)} q_0 q_1                                                      \\
		& =\frac{\gamma+1}{2\gamma} (q_0-\frac{\gamma-1}{\gamma+1}q_1)(q_0 - \frac{\gamma}{\gamma-1}q_1) \\
		& <0,
	\end{align*}
	since $q_0>q_1$ and, by \eqref{mu-gamma}, $\displaystyle q_0 < \frac{\gamma}{\gamma-1}q_1$.
	Hence the function $\displaystyle \frac{1}{F(u;\alpha)}$ is decreasing with respect to $\alpha$ and then $J(\alpha)$ is an increasing function.
	Therefore, for $0<\delta \leq 1$, the solution $\alpha_\mu$ is unique for sufficiently small $\mu>0$.
	Moreover, $\alpha_\mu \to 0$ as $\mu \to 0$ by the monotonicity of $J$ and \eqref{asym of J1}. That is, part (i) of Lemma \ref{exist-alpha_mu} holds.
	
	\emph{Step 3:} The existence of two solutions $\alpha_\mu$ as $\delta >1$. \par
	Since $\displaystyle \frac{1}{F(u;\alpha)}$ and
	$\displaystyle \frac{\partial}{\partial \alpha}\frac{1}{F(u;\alpha)}$ are continuous
	in $[q_1,q_0] \times (-\infty,0)$, the derivative of $J$ can be represented as
	\begin{align*}
		J'(\alpha)= \int_{q_0}^{q_1} \frac{\partial}{\partial \alpha}\frac{1}{F(u;\alpha)} du.
	\end{align*}
	
	By \eqref{derivative of F}, we claim that there exist two constants $-\infty<\alpha_M<\alpha_m<0$ such that 
	\begin{equation}\label{derivative-J}
		\begin{split}
			& J'(\alpha)>0, \text{ for }\alpha \in (\alpha_m, 0),\\
			& J'(\alpha)<0, \text{ for }\alpha \in (-\infty, \alpha_M).
		\end{split}
	\end{equation}
	That is, $ J(\alpha) $ is strictly increasing for $ \alpha \in (\alpha_m, 0) $ and strictly decreasing for $ \alpha \in (-\infty, \alpha_M) $. Then as $\mu>0$ is sufficiently small such that
$$\displaystyle \frac{1}{\mu}\left(\frac{R}{\gamma-1}\right)^{\delta}
> \max_{\alpha_M\leq\alpha\leq  \alpha_m} J(\alpha),$$
there exist exactly two solutions $\alpha_{\mu 1} \in (\alpha_m, 0)$ and $\alpha_{\mu 2} \in (-\infty, \alpha_M)$ to the equation \eqref{vp-alpha}, which yields the existence of two solutions $ (X_{\mu1}(u);\ \alpha_{\mu1}) $ and $ (X_{\mu2}(u);\ \alpha_{\mu2}) $ to \eqref{mu-non-inve}. 
Furthermore, by \eqref{asym of J1} \eqref{asym of J2}, it holds that
$\alpha_{\mu 1} \to 0$ and $\alpha_{\mu 2} \to -\infty$ as $\mu \to 0+$, which shows that part (ii) of Lemma \ref{exist-alpha_mu} holds.
Therefore, to complete the proof of Lemma \ref{exist-alpha_mu}, it remains to show that \eqref{derivative-J} holds.
	
	Indeed, as $\alpha \to 0-$, we can choose $\alpha<0$ with $\abs{\alpha}$ sufficiently small such that 
	$\displaystyle \alpha (\delta-1) g_1 g_2 < \frac{1}{2}g_1 f_2$. 
	Then we have 
	\begin{align*}
		J'(\alpha)>\int_{q_0}^{q_1}  
		\frac{(f_{2}(u)+\alpha g_{2}(u))^{\delta-1}}{(f_{1}(u) +\alpha g_{1}(u))^2}
		\big(\delta g_{2}(u)f_{1}(u)-\frac{1}{2}g_{1}(u)f_{2}(u)\big) du.
	\end{align*}
	Notice that $f_1(q_0)=f_1(q_1)=0$, there exists $\varepsilon>0$ such that 
	$\delta f_1(u) g_2(u) - \frac{1}{2} g_1(u) f_2(u) < \varepsilon_1 <0$ 
	for $u \in (q_1, q_1+\varepsilon) \cup (q_0-\varepsilon, q_0)$.
	Then 
	\begin{align*}
		J'(\alpha) & > -\left(\int_{q_1}^{q_1+\varepsilon}+\int_{q_0-\varepsilon}^{q_0} \right) 
		\frac{(f_{2}+\alpha g_{2})^{\delta-1}}{(f_{1} +\alpha g_{1})^2}
		\big(\delta g_{2}f_{1}-\frac{1}{2}g_{1}f_{2}\big) du \\
		& \quad -\int_{q_1+\varepsilon}^{q_0-\varepsilon} 
		\frac{(f_{2}+\alpha g_{2})^{\delta-1}}{(f_{1} +\alpha g_{1})^2}
		\big(\delta g_{2}f_{1}-\frac{1}{2}g_{1}f_{2}\big) du \\
		& \geq -\int_{q_0-\varepsilon}^{q_0} 
		\frac{f_{2}^{\delta-1}}{(f_{1} +\alpha g_{1})^2}
		\big(\delta g_{2}f_{1}-\frac{1}{2}g_{1}f_{2}\big) du \\
		& \quad -\int_{q_1+\varepsilon}^{q_0-\varepsilon} 
		\frac{(f_{2} - g_{2})^{\delta-1}}{f_{1}^2}
		\delta g_{2} f_{1}du \\
		& =: I_1 + I_2.
	\end{align*}
	By \eqref{formula of f_1}-\eqref{formula of g_2},
	\begin{align*}
		I_1 &= -\int_{q_0-\varepsilon}^{q_0} 
		\frac{u^2 f_{2}^{\delta-1}\big(\delta g_{2}f_{1}-\frac{1}{2}g_{1}f_{2}\big)}
		{((u-q_0)(u-q_1) + \alpha (\gamma u-(\gamma-1)q_0))^2} du \\
		& \geq \int_{q_0-\varepsilon}^{q_0}
		\frac{\varepsilon_0}{((q_0-q_1)(u-q_0)+ \alpha q_0)^2} du \\
		& \to +\infty  \quad \text{as $\alpha \to 0-$},
	\end{align*}
	since $-u^2 f_{2}^{\delta-1}\big(\delta g_{2}f_{1}-\frac{1}{2}g_{1}f_{2}\big)$ has a positive lower bound $\varepsilon_0$ for $u \in (q_0-\varepsilon, q_0)$. And $I_2$ is only a finite negative constant independent of $\alpha$ which implies that there exists a constant $\alpha_m < 0$ such that $J'(\alpha)>0$ for $\alpha \in (\alpha_m,0)$. 
	
	On the other hand, as $\alpha \to -\infty$, for a small constant $\varepsilon$, one can choose $\alpha<0$ with $ \abs{\alpha} $ sufficiently large such that $\alpha (\delta-1) g_1(u) g_2(u)- g_1(u) f_2(u) >0$ for $u \in (q_1, q_0-\varepsilon)$. Then 
	\begin{align*}
		J'(\alpha) &< \int_{q_1}^{q_0}
		\frac{(f_{2}+\alpha g_{2})^{\delta-1}}{(f_{1} +\alpha g_{1})^2}
		\big( g_1 f_2 - \alpha (\delta -1) g_1 g_2 \big)du \\
		&< \int_{q_0-\varepsilon}^{q_0} 
		\frac{(f_{2}+\alpha g_{2})^{\delta-1}}{(f_{1} +\alpha g_{1})^2}
		g_1 f_2 du\\
		&\quad - \int_{q_1}^{q_0-\varepsilon} 
		\frac{(f_{2}+\alpha g_{2})^{\delta-1}}{(f_{1} +\alpha g_{1})^2}
		\big(\alpha (\delta-1) g_1 g_2- g_1 f_2 \big) du \\
		&=: I_3 - I_4.
	\end{align*}
	Since $$\displaystyle \max_{u \in (q_1, q_0-2\varepsilon)} g_2(u)
	< \min_{u \in (q_0-\varepsilon, q_0)} g_2(u)
	\leq 0
	< \max_{u \in (q_1, q_0-2\varepsilon)} g_1(u) 
	< \min_{u \in (q_0-\varepsilon, q_0)} g_1(u),$$ 
	it holds that, for $\alpha<0$ with $ \abs{\alpha} $ sufficiently large, 
	\begin{align*}
		\min_{u \in (q_1, q_0-2\varepsilon)} \big(\alpha (\delta-1) g_1 g_2- g_1 f_2 \big)
		&> \max_{u \in (q_0-\varepsilon, q_0)} g_1 f_2, \\
		\min_{u \in (q_1, q_0-2\varepsilon)} 
		\frac{(f_{2}+\alpha g_{2})^{\delta-1}}{(f_{1} +\alpha g_{1})^2}
		&> \max_{u \in (q_0-\varepsilon, q_0)}
		\frac{(f_{2}+\alpha g_{2})^{\delta-1}}{(f_{1} +\alpha g_{1})^2},
	\end{align*}
	which implies $I_3 < I_4$. Therefore, there exists a constant $\alpha_M < 0$ such that
	$J'(\alpha) <0$ for $\alpha \in (-\infty, \alpha_M)$, which completes the proof for \eqref{derivative-J}.
	
\end{proof}


Notice that as $\alpha_\mu \to -\infty$, the pressure $p^\mu$ goes to infinity by \eqref{non-p-rep}. This means that in case $\delta >1$, the solution $U^\mu$ corresponding to the solution $\alpha_{\mu 2}$ in Lemma \ref{exist-alpha_mu} 
cannot converge to a solution of Euler system \eqref{aa}. For the cases of $\alpha_\mu \to 0$, we can follow the similar arguments in Section \ref{Section 3.2} and Section \ref{Section 4.1} to obtain the convergence of the corresponding solution to a shock solution to the Euler system \eqref{aa}. Therefore we omit the proof here and give the conclusion directly. 

\begin{lem}\label{last}
	Suppose that (\ref{mu-gamma}) holds. 
\begin{enumerate}
	\item[(i)] 	If $0 < \delta \leq 1$, there exists a small constant $\mu_1>0$ such that for any $\mu \in (0,\mu_1)$, the viscous solution $\{u^{\mu}\}$ to the problems \eqref{tem-non-first-eq}\eqref{nonsen-tem-first-bd} exists and it converges to  $u_{*}^{0}$ in $L^{1}(0,1)$ as $\mu \to 0$, where
	\begin{equation}
		u_{*}^{0}(x):=
		\left\{
		\begin{aligned}
			q_{0}, & \quad 0\leq x < X_{s}   \\
			q_{1}, & \quad  X_{s} < x \leq 1
		\end{aligned}
		\right.
	\end{equation}
	with 
	\begin{equation} 
		X_{s}:= \left( 1+ \left(\frac{\frac{\gamma+1}{\gamma-1} - \frac{q_1}{q_0}}{\frac{\gamma+1}{\gamma-1} - \frac{q_0}{q_1}}\right)^{\delta}
		\frac{q_1}{q_0} \right)^{-1} =\left( 1+\left(\frac{M_0}{M_1}\right)^{2\delta}
		\left(\frac{1-M_0^{-2}}{M_1^{-2}-1} \right)^{1+2\delta}
		\right)^{-1}.
	\end{equation}
	\item[(ii)] If $\delta > 1$, there exists a constant $\mu_2>0$ such that for any $\mu \in (0,\mu_2)$, there are two solutions $\{u^{\mu}_1\}$ and $\{u^{\mu}_2\}$ to the problems \eqref{tem-non-first-eq}\eqref{nonsen-tem-first-bd}, satisfying that, one of the solutions $\{u^{\mu}_1\}$ converges to $u_*^0$ in $L^{1}(0,1)$ as $\mu \to 0$ and the other solution $\{u^{\mu}_2\}$ cannot converge to a shock solution.

\end{enumerate}

\end{lem}
Obviously, Theorem \ref{thm: non-viscous} is an immediate consequence of Lemma \ref{last}.

\appendix

\section{Two Lemmas}

In this appendix, we give the following two lemmas which are employed in this paper. 

\begin{lem}\label{lemma1}
	Suppose $F \in C(\mathbb{R})$, $u\in C^1[0,1]$ satisfies 
	$$\partial_x u=F(u)$$ 
	with $u(0)>u(1)$.
	Then $u$ is a non-increasing function.
\end{lem}

\begin{proof}
	It suffices to prove that $\partial_x u(x) \leq 0$ for any $x \in [0,1]$.
	
	We claim that if $x_0,x_1 \in [0,1], x_0<x_1 $ with $u(x_0)>u(x_1)$, then we have $\partial_x u(x_0) \leq 0$ and $\partial_x u(x_1) \leq 0$.
	
	If not, assume $\partial_x u(x_0)=F(u(x_0))>0$.
	By the continuity of $u$ and $F$, it holds that $\partial_x u>0$ in a small neighborhood $(x_0, x_0+\delta)\subset (x_0,x_1)$ and then $u(x)>u(x_0)$ for any $x \in (x_0, x_0 +\delta)$.
	
	Define the level set
	$S_{z}:=\{x \in (x_0,x_1):u(x)=z \}.$
	Because $u(y) > u(x_0)$ for some $y \in (x_0, x_0+\delta)$ and $u(x_1) < u(x_0)$, there exists $\xi \in (y,x_1)$ such that $u(\xi)=u(x_0)$ so that $S_{u(x_0)}$ is not empty.
	
	Denote $x^*:=\inf S_{u(x_0)} \geq x_0+\delta$, then $u(x^*)=u(x_0)$ since $u$ is continuous. Thus $u(x)>u(x_0)$ for any $x\in(x_0,x^*)$ which implies 
	$$\partial_x u(x^*)=\lim_{x \rightarrow x^*-} \frac{u(x)-u(x^*)}{x-x^*} \leq 0.$$
	However $\partial_x u(x^*)=F(u(x^*))=F(u(x_0))>0$ which leads to a contradiction.
	This contradiction indicates that $\partial_x u(x_0) \leq 0$. 
	
	$\partial_x u(x_1)\leq 0$ can be proved similarly.	
	Moreover, if $u(x_0)<u(x_1)$, it can be similarly proved that $\partial_x u(x_0) \geq 0$ and $\partial_x u(x_1)\geq 0$.
	
	Finally, for any $x\in(0,1)$ satisfies $u(x)>u(0)$, it holds that $\partial_x u(x)\geq 0$ if we consider the interval $[0,x]$ and $\partial_x u(x) \leq 0$ if we consider the interval $[x,1]$.
	Hence $\partial_x u(x) = 0$ for any $x$ satisfies $u(x)>u(0)$. It can be similarly proved that $\partial_x u(x) = 0$ for any $x$ satisfies $u(x)<u(1)$. This implies that $u(1) \leq u(x) \leq u(0)$ and by the claim $\partial_x u(x) \leq 0$ for any $x \in [0,1]$.
\end{proof}

\begin{lem}\label{lemma2}
	Suppose $F \in C(\mathbb{R})$, $A \in C(\mathbb{R})$, $u\in C^1[0,1]$ satisfies 
	$$A(u)\partial_x u=F(u)$$ 
	with $u(0)>u(1)$.
	Assume the number of zeros of $A$ is finite and $A(y) \neq 0$ if $y \in [u(1),u(0)]$, then $u$ is a non-increasing function.
	
\end{lem}
\begin{proof}
	According to the proof of Lemma \ref{lemma1}, it suffices to prove that $u(1) \leq u(x) \leq u(0)$ for $x \in [0,1]$.
	
	If not, without loss of generality, assume that $ x_0\in(0,1) $ and $c_0=u(x_0)<u(1)$ is the minimum value of $ u $ in $[0,1]$.
	
	Since $A(x)$ only has finite zeros, one has that $A(y) \neq 0$ for $c_0 < y < c_0+\delta < u(1)$, with sufficiently small $\delta>0$. 
	For each $y \in (c_0,c_0 +\delta)$, let 
	$$S_{y}:=\{x \in [0,1]:\ u(x)=y \}.$$
	Define 
	\[\begin{split}
		x_1(y) &:= \sup_{0\leq x < x_0} S_{y},\\
		x_2(y) &:= \inf_{x_0<x\leq 1} S_{y}.
	\end{split}
	\]
	Obviously, $ x_1(y) < x_0 $, $ x_2(y) > x_0 $ and $ u(x_1(y)) = u(x_2(y)) = y $.
	Moreover, for any $ x\in (x_1(y),\ x_2(y))\setminus\set{x_0} $, it holds that $u(x) < y$ and $A(u(x)) \neq 0$.
	Hence, 
	$$0 \geq \partial_x u(x_1) \partial_x u(x_2) = \frac{F^2(y)}{A^2(y)} \geq 0.$$
	Therefore $\partial_x u(x)=0$ for any $x$ satisfies $u(x) \in (c_0,c_0+\delta)$.
	It contradicts with the assumption that the value of $u$ reach the minimum $c_0$ at $ x_0 $.
	
	Therefore, it holds that $u(1) \leq u(x) \leq u(0)$ and $u$ is a non-increasing function by Lemma \ref{lemma1}.
\end{proof}

\begin{rem}
	It can be easily seen from the proof of Lemma \ref{lemma2} that, in case $u(0)<u(1)$, then under the assumptions of Lemma \ref{lemma2} and $A(y) \neq 0$ for $y \in [u(0),u(1)]$, it holds that $u$ is a non-decreasing function.
\end{rem}

\section*{Acknowlegements}
The research of the paper was supported in part by the National Key R\&D Program of China (No. 2020YFA0712000), National Natural Science Foundation of China (No. 11971308), and the Fundamental Research Funds for the Central Universities.

\bibliographystyle{IEEEtran}
\end{document}